\documentclass[12pt]{amsart}

\usepackage{tgpagella}
\usepackage{euler}
\usepackage[T1]{fontenc}
\usepackage{amsmath, amssymb}
\usepackage[hidelinks]{hyperref}
\usepackage[english]{babel}
\usepackage{mathrsfs}
\usepackage{eucal}
\usepackage[all]{xy}
\usepackage{tikz}
\usepackage{mathtools}
\usepackage{multicol}

\newtheorem{thm}{Theorem}[subsection]
\newtheorem*{thm*}{Theorem}
\newtheorem{lem}[thm]{Lemma}
\newtheorem*{prob*}{Problem}
\newtheorem{fact}[thm]{Fact}

\newtheorem{prop}[thm]{Proposition}
\newtheorem*{prop*}{Proposition}
\newtheorem{conj}[thm]{Conjecture}
\newtheorem{cor}[thm]{Corollary}
\newtheorem*{cor*}{Corollary}

\theoremstyle{definition}
\newtheorem{defn}[thm]{Definition}
\newtheorem*{defn*}{Definition}

\newtheorem{remark}[thm]{Remark}

\newtheorem{question}[thm]{Question}

\newtheorem*{question*}{Question}
\newtheorem*{Pquestion*}{Popa's question}

\newtheorem*{conv*}{Convention}

\newtheorem{innercustomthm}{Theorem}
\newenvironment{customthm}[1]
  {\renewcommand\theinnercustomthm{#1}\innercustomthm}
  {\endinnercustomthm}
\newtheorem{innercustomcor}{Corollary}
\newenvironment{customcor}[1]
  {\renewcommand\theinnercustomcor{#1}\innercustomcor}
  {\endinnercustomcor}
\newtheorem{innercustomprop}{Proposition}
\newenvironment{customprop}[1]
  {\renewcommand\theinnercustomprop{#1}\innercustomprop}
  {\endinnercustomprop}

\def\bf{\mathbf}

\def\bd{\mathbf}

\def\cal{\mathcal}

\def\u{\mathsf 1}

\def \ggr{G_{\operatorname{GR}}}

\def \Gsm{\cal G_{\operatorname{sm}}}
\def \Gam{\cal G_{\operatorname{am}}}
\def \Gll{\cal G_{\operatorname{ll}}}
\def \Gamll{\cal G_{\operatorname{am},\operatorname{ll}}}
\def \Gsmll{\cal G_{\operatorname{sm},\operatorname{ll}}}
\def \Gamw{\cal G_{\operatorname{am},w}}
\def \gos{G_{\operatorname{OS}}}
\def \ec{\operatorname{ec}}

\makeatletter

\def\dotminussym#1#2{%
  \setbox0=\hbox{$\m@th#1-$}%
  \kern.5\wd0%
  \hbox to 0pt{\hss\hbox{$\m@th#1-$}\hss}%
  \raise.6\ht0\hbox to 0pt{\hss$\m@th#1.$\hss}%
  \kern.5\wd0}

\def \u{\mathcal U}

\newcommand{\mc}{\mathcal}

\begin{document}
\begin{abstract}

     We introduce and study Polish topologies on various spaces of countable enumerated groups, where an enumerated group is simply a group whose underlying set is the set of natural numbers.
    Using elementary tools and well-known examples from combinatorial group theory, combined with the Baire category theorem, we obtain a plethora of results demonstrating that several phenomena in group theory are generic. In effect, we provide a new topological framework for the analysis of various well known problems in group theory. We also provide a connection between genericity in these spaces, the word problem for finitely generated groups and model-theoretic forcing. Using these connections, we investigate the natural question: when does a certain space of enumerated groups contain a comeager isomorphism class? We obtain a sufficient condition that allows us to answer the question in the negative for the space of all enumerated groups and the space of left orderable enumerated groups.  We document several open questions in connection with these considerations.
\end{abstract}
\title{Generic algebraic properties in spaces of enumerated groups}
\author{Isaac Goldbring, Srivatsav Kunnawalkam Elayavalli and Yash Lodha}
\thanks{I. Goldbring was partially supported by NSF  grant DMS-2054477. Y. Lodha was supported by a CMC fellowship at the Korea institute of advanced study, and a START-preis grant of the Austrian science fund.}

\address{Department of Mathematics\\University of California, Irvine, 340 Rowland Hall (Bldg.\# 400),
Irvine, CA 92697-3875}
\email{isaac@math.uci.edu}
\urladdr{http://www.math.uci.edu/~isaac}

\address{Department of Mathematics\\Vanderbilt University, 1326 Stevenson Center, Station B 407807, Nashville, TN 37240}
\email{srivatsav.kunnawalkam.elayavalli@vanderbilt.edu}
\urladdr{https://sites.google.com/view/srivatsavke}

\address{Faculty of mathematics, University of Vienna}
\email{yashlodha763@gmail.com}
\urladdr{https://yl7639.wixsite.com/website}

\maketitle

\section{Introduction}

This paper aims to contribute to the endeavour of studying the theory of countable groups from a topological lens. We are interested in the setting of \textbf{enumerated groups}, where an enumerated group is simply a group structure on the set $\mathbf{N}$ of natural numbers. (One may view this also as a countable group endowed with a fixed bijection with $\bf N$.) Equipped with a natural topology (constructed in Section \ref{Polish}), the set of all enumerated groups forms a Polish space. The space of enumerated groups is very natural from the point of view of first-order logic in that it is simply the space of countably infinite $L$-structures in the case that $L$ is the usual first-order language of groups. In group theoretic language, basic open sets in this topology are exactly the sets of all enumerated groups satisfying a given finite system of equations and inequations. It is imperative to caution the reader early on that this space is notably different from the usual topology that group theorists consider on the space of (finitely generated and marked) groups, namely, the \textbf{Grigorchuk space of marked groups}. 

Rather than embarking on a study of this space itself, we isolate and study a large family of relevant subspaces.
Given one among a specified list of properties of countable groups, we show that the subspace topology endows the family of enumerated groups satisfying this property with the structure of a Polish space. In this paper, we consider the following properties:

 \begin{multicols}{2}
    \begin{itemize}
        \item Amenability
        \item No $\mathbf{F}_2$ subgroups
        \item left orderability
        \item local indicability
        \item biorderability
        \item unique product property
        \item torsion-free
        \item soficity
        \item does not satisfy a law
    \end{itemize}
    \end{multicols}

In the sequel, we denote the set of these group theoretic properties by $\mathcal{P}$.
Given a property $P\in \mathcal{P}$, we let $\mathcal{G}_P$ denote the subspace of $\mathcal{G}$ consisting of enumerated groups which satisfy $P$. We will show that, for each $P\in \mathcal{P}$, the space $\mathcal{G}_P$, endowed with the subspace topology, is a Polish space.

Our main interest in Polish spaces is that the Baire category theorem applies to such spaces.
Recall that the Baire category theorem states that, in any Polish space $X$, the intersection of countably many dense open subsets of $X$ is once again a dense subset of $X$. In the language of Baire category, an intersection of countably many dense open sets is called \textbf{comeager}, and if a certain property holds for all elements of a comeager subset of the space, it is natural to say that the property is \textbf{generic} in this space. 

This article is centered around the following question:

\begin{question}\label{Question: main}
Fix $P\in \mathcal{P}$.
\begin{enumerate}
\item What group-theoretic properties are generic in $\mathcal{G}_P$?
\item Is there is a comeager set $\mathcal{X}_P\subset \mathcal{G}_P$ such that all groups in $\mathcal{X}_P$ are isomorphic?\footnote{We would like to point out that after seeing an old version of this article, D. Osin asked us the this question in the context of $P$ being amenability. This is what inspired us to work on this problem in this generality. Osin's original question on amenability however remains a difficult open problem.}
\end{enumerate}
\end{question}

 The second part of the above question turns out to be quite difficult, and much of the work done in this article will illustrate why this is the case.
We provide a partial answer to this in Theorem \ref{main theorem 4}. This settles the question in the negative for the space of all enumerated groups, and also for the space of all left orderable enumerated groups (see Corollary \ref{cor:leftorderablemeager}).
In connection to the first part of the above question, we indeed demonstrate that a plethora of group theoretic phenomena are in fact generic in $\mathcal{G}_P$. 

A rather elementary, yet important, feature of all the  properties in $\mathcal{P}$ is that they are closed under direct sums and directed unions. Moreover, all properties in $\mathcal{P}$ (besides not satisfying a law) are inherited by subgroups. Our conceptual recipe may have applications for various group theoretic properties not considered here which also share these features.

To establish some of our results, an analogy is also drawn between the Polish space of enumerated groups and the aforementioned Grigorchuk space of marked groups; in this connection, various elementary tools from combinatorial group theory are used.
The Grigorchuk space of finitely generated marked groups is the space $$\mathcal{M}=\bigcup_{n\in \mathbf{N}}\mathcal{M}_n$$ where $\mathcal{M}_n$ is the set of marked groups consisting of pairs $(G,S)$, where $G$ is a group endowed with an ordered generating set $S$ of cardinality $n$.
(We will recall the topology on $\mathcal{M}$ in Subsection \ref{marked}.)
A property $P\in P$ is said to be an \textbf{open} (respectively \textbf{closed}) property if the set of groups that satisfies $P$ forms an open (respectively closed) set in $\mathcal{M}$.

A finitely generated group $G$ is said to be an \textbf{isolated group}, if some (equivalently, any) marking $(G,S)$ of $G$ is an isolated point in $\mathcal{M}$. (Note that isolated groups are finitely presentable.)  Examples of isolated groups include finite groups and finitely presented simple groups.

In this paper, we will need to work with a generalization of the notion of an isolated group.  A finitely generated group $G$ is said to be \textbf{$P$-near isolated} (where $P\in \mathcal{P}$) if there is a finitely presented group $H$ and a finite subset $X\subset H$ such that:
\begin{enumerate}
\item There is a surjective homomorphism $\phi:H\to K$ that is injective on $X$ and for which $K$ has property $P$.
\item For any surjective homomorphism $\phi:H\to K$ that is injective on $X$ and for which $K$ has $P$, we have that $K$ contains a subgroup isomorphic to $G$.
\end{enumerate}

Note that an isolated group $G$ is $P$-near isolated for $P$ being the tautologically true property; simply take $H=G$ and let the finite subset $X$ be the open ball in the Cayley graph that isolates $G$.

\begin{remark}   In this article, we will encounter several examples of $P$-near isolated groups for various $P\in \mathcal{P}$.  
These examples will play a key role in establishing the genericity of various properties in the relevant spaces.
The following are two natural situations when $G$ is $P$-near isolated:
\begin{enumerate}
\item $G$ embeds in a simple subgroup of a finitely presented group that satisfies $P$.
\item $G$ embeds in an isolated group that satisfies $P$.
\end{enumerate}
\end{remark}

\subsection{The subgroup structure of a "generic group".}

Our first collection of results center around the subgroup structure of generic groups in the relevant spaces: 

\begin{thm}\label{Theorem: main1}
Let $P\in \mathcal{P}$ be a property. Then the following hold:
\begin{enumerate}
\item There is a comeager set $\mathcal{X}_P\subset \mathcal{G}_P$ such that, for every enumerated group $G\in \mathcal{X}_P$ and every $P$-near isolated group $H$, $G$ contains a subgroup isomorphic to $H$.
\item Let $Q$ be an open property of finitely generated groups for which there is a finitely generated group that satisfies both $Q$ and $P$. Then there is an open dense set $\mathcal{X}_P\subset \mathcal{G}_P$ such that every enumerated group $G\in \mathcal{X}_P$ contains a finitely generated subgroup satisfying $Q$.

\end{enumerate}
\end{thm}

In order to state our next result, we need a couple more definitions.  We say that a property $P\in \mathcal{P}$ is a \textbf{Boone-Higman} property if every group $G$ that has a solvable word problem and satisfies $P$ embeds in a simple subgroup of a finitely presented group $H$ that also satisfies $P$.  (The reason for the terminology is the Boone-Higman theorem, which states that the tautologically true property is a Boone-Higman property.)
Several of the properties that we consider in this article are Boone-Higman properties, most notably left orderability (which was demonstrated in a beautiful paper of Bludov and Glass \cite{bludovglass}).

Given a property $P\in \mathcal{P}$, we say that $P$ is \textbf{strongly undecidable} if there exists a finitely presented group $G$ and a finite subset $X\subset G$ such that:
\begin{enumerate}
\item There is a group $H$ satisfying $P$ and a surjective homomorphism $\phi:G\to H$ whose restriction to $X$ is injective.
\item If $H$ is a group satisfying $P$ for which there exists a surjective homomorphism $\phi:G\to H$ whose restriction to $X$ is injective, then $H$ has an unsolvable word problem.
\end{enumerate}
We first note that the tautologically true property is strongly undecidable, that is, there is indeed an example of a finitely presented group all of whose nonidentity quotients have an unsolvable word problem. (This is the main theorem in \cite{CFMiller}.) 



\begin{thm}\label{Theorem: main2}
Let $P\in \mathcal{P}$ be a Boone-Higman property (for instance, left orderability). Then there is a comeager set $\mathcal{X}_P\subset \mathcal{G}_P$ such that every group in $\mathcal{X}_P$ contains an isomorphic copy of every finitely generated group satisfying $P$ with solvable word problem.
\end{thm}

On the other hand, the problem of determining the finitely generated subgroup structure of a generic group in $\mathcal{G}_{lo}$ is reduced to the following:

\begin{thm}\label{Theorem: main3}
Let $P\in \mathcal{P}$ be a Boone-Higman property that is inherited by subgroups. Then exactly one of the following holds:
\begin{enumerate}

\item $P$ is not strongly undecidable.  In this case, there is a comeager set $\mathcal{X}_{P}\subset \mathcal{G}_{P}$ such that for each group $G\in \mathcal{X}_{P}$, the set of finitely generated subgroups of $G$ coincides with the set of finitely generated groups satisfying $P$ that also have a solvable word problem.

\item $P$ is strongly undecidable.  In this case, there is a comeager set $\mathcal{X}_P\subset \mathcal{G}_P$ such that, for each group $G\in \mathcal{X}_P$, the set of finitely generated subgroups of $G$ contains all finitely generated groups satisfying $P$ that also have a solvable word problem, but also contains a finitely generated subgroup with $P$ that has an unsolvable word problem.

\end{enumerate}

\end{thm}

\begin{remark}\label{Remark: StronglyUndecidable}
Since the tautologically true property is a strongly undecidable Boone-Higman property, part $(2)$ of the previous theorem \ref{Theorem: main2} provides a comeager set $\mathcal{X}\subset\mathcal{G}$ such that every isomorphism type in $\mathcal{X}$ contains a finitely generated subgroup with an unsolvable word problem.
\end{remark}

\begin{remark} 
Since a natural way of distinguishing isomorphism types of countable groups is to distinguish the finitely generated subgroup structure, Theorem \ref{Theorem: main3} illustrates the difficulty of answering the second part of Question \ref{Question: main}.
\end{remark}

Using Theorem \ref{Theorem: main3} and techniques from model theory (see Hodges \cite{hodges}), we are able to prove the following theorem, which provides a road map to answer to Question \ref{Question: main}:(2) for some  interesting subspaces: 

\begin{thm}\label{main theorem 4}
If $P$ is a strongly undecidable, recursively axiomatizable, Boone-Higman property that is closed under subgroups, then there is no comeager isomorphism class in $\mathcal{G}_P$.  
\end{thm}

We verify the fact that left orderability is recursively axiomatizable in Proposition \ref{leftrec}. Morever, we also present an argument of A. Darbinyan showing that left orderability is a strongly undecidable property in Proposition \ref{Arman's proof}. Hence, we obtain: 

\begin{cor}\label{cor:leftorderablemeager}
The space of left orderable enumerated groups does not contain a co-meager isomorphism class.
\end{cor}

\subsection{The algebraic structure and first order theory of a "generic group".}
Our next result examines the algebraic structure of generic groups in some of the relevant spaces.
Recall that a group $G$ is said to be \textbf{verbally complete} if, for each freely reduced word $W(x_1,...,x_n)$ in the letters $x_1^{\pm 1},...,x_n^{\pm 1}$ (for an arbitrary $n\in \mathbf{N}$)
and each element $f\in G$, there are elements $g_1,...,g_n\in G$ such that $$W(g_1,...,g_n)=f$$ 
Verbally complete groups are in particular \textbf{divisible}: for each $f\in G$ and $n\in \mathbf{N}\setminus \{0\}$, there is a $g\in G$ such that $g^n=f$.
We use elementary tools from combinatorial group theory to deduce the following, which applies in particular in the case when $P\in \{\text{left orderable},\text{locally indicable}, \text{torsion-free}\}$.

\begin{thm}\label{Theorem ConjOrd}
Let $P\in \mathcal{P}$ be a property of torsion-free groups that is closed under amalgamation along infinite cyclic subgroups and HNN extensions with associated subgroups that are infinite cyclic. 
Then there is a comeager set $\mathcal{X}_P\subset \mathcal{G}_P$ such that each $G\in \mathcal{X}_P$ has only one nontrivial conjugacy class (in particular, it is simple) and is verbally complete.
\end{thm}

A fundamental notion from logic relevant to our considerations is that of \textbf{elementary equivalence}: two groups are elementary equivalent if they have the same first-order theory.
We next discuss the question of when an amenable group can have the same first-order theory as a nonamenable group or a group with property (T),
and prove the following:

\begin{thm}\label{firstorder}
There is a comeager set $\mathcal{X}\subset \mathcal{G}_{am}$ such that, for each $G\in \mathcal{X}$, the following holds:
\begin{enumerate}
\item There are continuum many pairwise nonisomorphic countable nonamenable groups with the same first-order theory as $G$.
\item $G$ cannot have the same first-order theory as a group with property (T).
\end{enumerate}
\end{thm}

Another key concept we investigate is the notion of locally universal groups.
Let $P\in \mathcal{P}$ be a property and consider the Polish space $\mathcal{G}_P$ as above.
An enumerated group $H\in \mathcal{G}_P$ is \textbf{locally universal for $\mathcal{G}_P$} if any group in $\mathcal{G}_P$ embeds into an ultrapower of $H$.
We denote the set of all locally universal groups in $\mathcal{G}_P$ by $\mathcal{G}_{lu,P}$.

\begin{thm}\label{locunivcomeager}
For each $P\in \mathcal{P}$, $\mathcal{G}_{lu,P}$ is a comeager subset of $\mathcal{G}_P$.
\end{thm}

Since comeager sets in a Polish space are closed under countable intersection, we may assume that all the comeager sets defined in this introduction consist of locally universal groups and thus satisfy the conclusions of all the relevant theorems.
 
 \subsection{Applications}

The above theorems allow us to deduce a plethora of genericity phenomena.
A nice illustration of the consequences of Theorem \ref{Theorem: main1} is an application to
the \textbf{von Neumann-Day Problem}, which refers to the conjunction of two problems about the class of amenable groups. 
For a detailed survey on the notion of amenability for countable, discrete groups, we refer the reader to \cite{bartholdi}.

\begin{prob*}[The von Neumann-Day Problem]
Must a group without nonabelian free subgroups be amenable?  Must an amenable group be elementary amenable?
\end{prob*}


The first problem was resolved by Olshanskii in 1980 \cite{olshanskii} and the second problem was resolved by Grigorchuk in 1984 \cite{grigorchuk84}.
In 1998, Grigorchuk \cite{grigorchuk98} provided the first finitely presented counterexample to the second problem.
In \cite{lodhamoore}, the third author with Moore constructed the first torsion-free finitely presented counterexample to the first problem. 
Since the examples of Lodha-Moore and Grigorchuk are easily seen to be isolated groups, we obtain the following striking corollary to Theorem \ref{Theorem: main1}:

\begin{cor}\label{Corollary VNDay}
The following holds:
\begin{enumerate}
\item The generic enumerated group without $\mathbf{F}_2$ subgroups is nonamenable.
\item The generic left orderable enumerated group without $\mathbf{F}_2$ subgroups is nonamenable. 
\item The generic enumerated amenable group is not elementary amenable.
\end{enumerate}
In fact, we can choose these comeager sets to be open dense sets.
\end{cor}

The following was pointed out to us by Denis Osin and stands in contrast to the previous corollary:
\begin{prop}\label{remark VNDay}

\

\begin{enumerate}
\item The generic enumerated group without $\mathbf{F}_2$ subgroups is inner amenable.
\item The generic left orderable enumerated group without $\mathbf{F}_2$ subgroups is inner amenable. 
\end{enumerate}
\end{prop}
Our next set of applications concerns the class of orderable groups.
Endowing groups with order structures has been an important theme in modern group theory.
There is a deep and striking interplay between notions of orderability and the topology and dynamics of group actions.
For instance, whether a countable group admits a faithful action by orientation preserving homeomorphisms on the real line admits a surprisingly elementary algebraic characterisation:
such an action exists if and only if the group is left orderable.
Related notions of orderability include: local indicability, biorderability, and the unique product property.
(All of these notions will be defined in Section \ref{orderability}.)
One has the following inclusions: $$
\text{bi-orderable groups}\subsetneq \text{locally indicable groups}\subsetneq \text{left orderable groups}$$ $$\subsetneq \text{groups with the unique product property}\subsetneq \text{ torsion-free groups}.$$
Note that in each case, only a handful examples of groups that witness that the inclusions are proper are known.
We deduce the following corollaries of Theorem \ref{Theorem: main1}:

\begin{cor}\label{Corollary LeftOrderable}
There is a comeager set $\mathcal{X}_{lo}\subset \mathcal{G}_{lo}$ such that each $G\in \mathcal{X}_{lo}$ satisfies:
\begin{enumerate}
\item It is not locally indicable.
\item It does not have the Haagerup property.
\item It does not admit nontrivial actions by $C^1$-diffeomorphisms on the closed interval or the circle.
\item Contains an isomorphic copy of every finitely generated left orderable group with a solvable word problem.
\end{enumerate}
\end{cor}

\begin{cor}\label{Corollary LocallyIndicable}
There is a comeager set $\mathcal{X}_{li}\subset \mathcal{G}_{li}$ such that each $G\in \mathcal{X}_{li}$ satisfies:
\begin{enumerate}
\item It is not biorderable.
\item It does not admit nontrivial actions by $C^1$-diffeomorphisms on the closed interval, $[0,1)$ or the circle.
\end{enumerate}
\end{cor}
A well known conjecture of Peter Linnel asserts that every left orderable group is either locally indicable or else contains a nonabelian free subgroup.
Corollary \ref{Corollary LeftOrderable}$(4)$ suggests why it is difficult to find counterexamples to this conjecture as the generic left orderable group contains a nonabelian free subgroup.  
Recall that the Witte-Morris theorem states that every left orderable amenable group is locally indicable, and the Thurston stability theorem states that any group of $C^1$-diffeomorphisms of an inteval of the form $[x,y)$ or $(x,y]$ is locally indicable. Corollary \ref{Corollary LeftOrderable} also implies that the converses to both these statements fail generically.

Kaplansky made the following conjectures in the $1960$'s. (See \cite{gardam} for a brief historical survey.)

\begin{conj}
Let $G$ be a torsion-free group and let $K$ be a field. Consider the group ring $K[G]$.
\begin{enumerate}

\item (Kaplansky's unit conjecture) Every unit in $K[G]$ is of the form $kg$ for $k\in K,g\in G$.

\item (Kaplansky's zero divisor conjecture) $K[G]$ has no non-trivial zero divisors.

\item (Kaplansky's idempotent conjecture) $K[G]$ has no idempotents other than $0
$ and $1$.

\end{enumerate}
\end{conj}

It is known that a positive solution to the unit conjecture implies a positive solution to the zero divisor conjecture, which in turn implies a positive solution to the idempotent conjecture.  It is also known that groups with the unique product property satisfy the unit conjecture, although this conjecture has a negative solution in general.
We note the following consequences of our results for the Kaplansky conjectures as it relates to the unique product property.
The first is the "generic" version of Gardam's recent breakthrough counterexample to the unit conjecture (\cite{gardam}).

\begin{cor}\label{Corollary Kaplansky}
The following holds:
\begin{enumerate}
\item There is an open dense set $\mathcal{X}\subset \mathcal{G}_{tf}$ such that each enumerated group $G\in \mathcal{X}$ is a counterexample to the unit conjecture.
\item The generic torsion-free enumerated group does not have the unique product property.
\item The generic group with the unique product property is not left orderable.
\item Either the Kaplansky zero divisor conjecture holds or else the generic torsion-free group does not satisfy the zero divisor conjecture.
\item Either the Kaplansky idempotent conjecture holds or else the generic torsion-free group does not satisfy the idempotent conjecture.
\end{enumerate}
\end{cor}

The notion of a sofic group was introduced by Gromov as a generalisation of both amenable and residually finite groups. The property of soficity is of considerable interest because it implies several important general conjectures of group theory.
(We direct the reader to \cite{pestov_2008} for a survey.)
While it remains open whether there is a nonsofic group, we note the following consequence for soficity in the context of the spaces $\mathcal{G}$ and $\mathcal{G}_P$ (for $P\in \mathcal{P}$):

\begin{cor}\label{Corollary Sofic}
For any $P\in \mathcal{P}$, either all groups in $\mathcal{G}_P$ are sofic or else the generic group in $\mathcal{G}_P$ is nonsofic.
\end{cor}

\begin{remark}
The special case of the previous corollary when $P$ is the tautologically true property was observed by Glebsky in \cite{gleb17}.
\end{remark}

In the last section of the paper, we provide an analysis of Question \ref{Question: main} part $(1)$, and collect some partial answers.
Moreover, we study certain natural questions that emerge in the setting of Polish spaces consisting of enumerated groups satisfying a given law, and provide some partial answers. We also document a connection of these considerations with the von Neumann-Day problem for ultrapowers.

\section{Preliminaries}
\subsection*{Conventions and Notations} 
In this paper, we use $\bf N$ to denote the set of positive natural numbers, that is, ${\bf N}=\{1,2,3,\ldots\}$.\footnote{Apologies to the logicians for this notation, but it makes a lot of our expressions cleaner to read.}  Given $n\in \bf N$, we also set $[n]=\{1,\ldots,n\}$.

Given a (group-theoretic) word $w(x_1,\ldots,x_n)$, we call $n$ the \textbf{arity} of the word and denote it by $n_w$.  By a \textbf{system} we mean a finite system of equations and inequations of the form $w=e$ or $w\not=e$, where $w$ is a word.  We use letters such as $\Sigma$ and $\Delta$ (sometimes with accents or subscripts) to denote systems.  If each word in the system has its variables amongst $x_1,\ldots,x_n$, then we write $\Sigma(x_1,\ldots,x_n)$ and extend the notion of arity to systems in the obvious way, using the notation $n_\Sigma$.  If $\Sigma(\vec x,\vec y)$ is a system, $G$ is a group, and $\vec a$ is a tuple from $G$ with the same length as $\vec x$, then we can consider the system $\Sigma(\vec a,\vec y)$, which we call a system \textbf{with coefficients}.  Given a system $\Sigma$, an enumerated group $G$, and $\vec a\in G^{n_\Sigma}$, we write $G\models \Sigma(\vec a)$ to denote that the system is true in $G$ when $\vec a$ is plugged in for $\vec x$.  
\subsection{Ultraproducts of groups}

Given a set $I$, an \textbf{ultrafilter on $I$} is a $\{0,1\}$-valued finitely additive probability measure $\u$ defined on all subsets of $I$.  One often conflates an ultrafilter $\u$ with its collection of measure $1$ sets, thus writing $A\in \u$ rather than $\u(A)=1$.  An ultrafilter $\u$ on $I$ is called \textbf{nonprincipal} if all finite sets have measure $0$.  A straightforward Zorn's lemma argument shows that nonprincipal ultrafilters exist on any infinite set.

Now suppose that $(G_i)_{i\in I}$ is a family of groups and that $\u$ is an ultrafilter on $I$.  The \textbf{ultraproduct of the family $(G_i)$ with respect to $\u$} is the group $$\prod_\u G_i:=(\prod_{i\in I} G_i)/N$$ where $N$ is the normal subgroup of $\prod_{i\in I}G_i$ given by $$N=\{g\in \prod_{i\in I}G_i\mid \text{ for which }\{i\in I\mid g(i)=e_{G_i}\}\in \mathcal{U}\}$$  Given $g\in \prod_{i\in I}G_i$, we denote its coset in $\prod_\u G_i$ by $g_\u$.  Thus, $g_\u=h_\u$ if and only if $\{i\in I\mid g(i)=h(i)\}\in \u$.  Given any word $w(\vec x)$ and $\vec a_\u\in (\prod_\u G_i)^{n_w}$, note that $w(\vec a_\u)=(w(\vec a(i))_\u$, whence $$\prod_\u G_i\models w(\vec a_\u)=e$$ if and only if $$\{i\in I\mid G_i\models w(\vec a(i))=e_{G_i}\}\in \u$$

When each $G_i=G$, we speak of the \textbf{ultrapower} $G^\u$ of $G$ with respect to $\u$.  The map which sends $g\in G$ to the coset of the sequence constantly equal to $g$ is called the \textbf{diagonal embedding} of $G$ into $G^\u$.

\subsection{Some model theory of groups}

A \textbf{quantifier-free formula} is a finite disjunction of systems.\footnote{This abuse of terminology is justified by the existence of disjunctive normal form.} A formula $\varphi(\vec x)$ is an expression of the form $$Q_1x_1\cdots Q_m x_m\psi(\vec x,\vec y)$$ with $\psi$ quantifier-free and each $Q_i\in \{\forall,\exists\}$.\footnote{This abuse of terminology is justified by the existence of prenex normal form.}  Given a formula $\varphi(\vec x)$, a group $G$, and a tuple $\vec a\in G^{n_\varphi}$, the definition of $\varphi(\vec a)$ being true in $G$, denoted $G\models \varphi(\vec a)$, is defined in the obvious way.  A formula without any free variables is called a sentence and is either true or false in a given group.  For example, $$G\models \forall x\forall y\exists z (   x=e\vee y=e\vee x^2\not=e\vee y^2\not= e\vee z^{-1}xzy^{-1}=e)$$ if and only if any two elements of $G$ of order 2 are conjugate. 

The following fundamental fact is known as \textbf{\L os' theorem} or the \textbf{Fundamental theorem of ultraproducts}:

\begin{fact}
For any family $(G_i)_{i\in I}$ of groups, any ultrafilter $\u$ on $I$, any formula $\varphi(\vec x)$, and any $\vec a\in \prod_\u G_i$, we have
$$\prod_\u G_i\models \varphi(\vec a)\Leftrightarrow G_i\models \varphi(\vec a(i)) \text{ for }\u\text{-almost all }i\in I.$$
\end{fact}

Groups $G$ and $H$ are called \textbf{elementarily equivalent}, denoted $G\equiv H$, if, given any sentence $\sigma$, we have $G\models \sigma$ if and only if $H\models \sigma$.  It follows from \L o's theorem that any group is elementarily equivalent to any of its ultrapowers.  Although we will not need it in this paper, the \textbf{Keisler-Shelah theorem} shows that elementary equivalence can be given a completely group-theoretic reformulation, namely two groups are elementarily equivalent if and only if they have isomorphic ultrapowers.



A set of sentences is called a \textbf{theory}.  If $T$ is a theory, we write $G\models T$ to indicate that $G\models \sigma$ for all $\sigma\in T$.  A class $\frak C$ of groups is called \textbf{elementary} (or \textbf{axiomatizable}) if there is a theory $T$ such that, for any group $G$, $G\in \frak C$ if and only if $G\models T$; in this case, we call the theory $T$ a set of \textbf{axioms} for the class.  For example, the classes of abelian groups and nilpotent class 2 groups are elementary.  

A sentence is called \textbf{universal} if, using the above notation, $Q_i=\forall$ for all $i=1,\ldots,m$.  A theory is called universal if it consists only of universal sentences.  An elementary class is called \textbf{universally axiomatizable} if it has a universal set of axioms.  The following is a special case of a more general test for axiomatizability of a class of groups:

\begin{fact}
A class of groups is universally axiomatizable if and only if it is closed under isomorphism, ultraproducts, and subgroups.
\end{fact}

Occasionally we will need to leave the confines of first-order logic and speak of infinitary formulae.  The class of $L_{\omega_1,\omega}$ formulae is the extension of the collection of all formulae obtained by allowing countable conjunctions and disjunctions rather than merely finite conjunctions and disjunctions.

\section{The Polish space of enumerated countable groups}\label{Polish}

\subsection{Introducing the space}  
By an \textbf{enumerated group}, we mean a group whose underlying set is $\mathbf{N}$. 
  We let $\cal G$ denote the set of enumerated groups and we let $\frak G$ denote the class of all isomorphism classes of countable groups.  We let $\rho:\cal G\to \frak G$ denote the obvious ``reduction'' with the convention that we write $G$ instead of $\rho(\bd G)$ (which is a bit abusive as we are conflating the difference between a group and its isomorphism class).  We adopt a similar convention with subsets of $\cal G$:  if $\cal C$ is a subset of $\cal G$, then we write $\frak C$ for the image of $\cal C$ under $\rho$.    In particular, for any property $P$, the set of isomorphism types of $\mathcal{G}_P$ is denoted by $\mathfrak{G}_P$.  We call $\cal C\subseteq \cal G$ \textbf{saturated} if $\rho^{-1}(\rho(\cal C))=\cal C$; in other words, $\cal C$ is saturated if it is closed under relabeling of elements.  Note that all sets of the form $\cal G_P$ are saturated.

  For ease of notation, in most cases we shall denote an enumerated group and its isomorphism type by the same symbol, as it is usually the case that which one is being considered will be made clear from the context.
  However, in certain situations, we may denote the group by a letter such as $G$ and its (chosen) enumeration as $\bd G$.

To each enumerated group $G$, we assign the associated multiplication function $\mu_\bd G:\bf N\times \bf N\to \bf N$, inversion function $\iota_\bd G:\bf N\to \bf N$, and identity element $e_\bd G\in \bf N$.  Consequently, we identify each element of $\cal G$ with a unique element of the zero-dimensional Polish space $\cal X:=\bf N^{\bf N\times \bf N}\times \bf N^\bf N\times \bf N$.

\begin{prop}
$\cal G$ is a closed subspace of $\cal X$. Consequently, with the induced topology, $\cal G$ is a zero-dimensional Polish space.\footnote{Although all of the information about $\bd G$ is contained in the multiplication map $\mu_\bd G$, if we identified $\bd G$ with $\mu_\bd G$, the resulting subspace of $\bf N^{\bf N\times \bf N}$ would not be Polish but rather $\boldsymbol{\Sigma}^0_3$.}  
\end{prop}

\begin{proof}
It suffices to observe that $\cal G$ is the intersection of the following closed subsets of $\cal X$:
\begin{enumerate}
    \item $\bigcap_{m,n,p\in \bf N}\{(f,g,a)\in \cal X \ : \ f(f(m,n),p)=f(m,f(n,p))\}$
    \item $\bigcap_{m\in \bf N}\{(f,g,a) \in \cal X \ : \ f(m,a)=f(a,m)=m\}$
    \item $\bigcap_{m\in \bf N}\{(f,g,a) \in \cal X \  : \ f(m,g(m))=f(g(m),m)=a\}$
    \item $\bigcap_{m\in \bf N, n\in \bf N\setminus \{m\}}\{(f,g,a) \in \cal X \  : \ f(m,g(n))\neq a  \}$
\end{enumerate}
\end{proof}

It will be convenient to recast the induced topology on $\cal G$ in more group-theoretic terms.  We let ${\bf W}$ denote the set of expressions of the form $w(\vec a)$, where $w(\vec x)$ is a word and $\vec a\in \mathbf{N}^{n_w}$.  Given an enumerated group $G$ and $w\in {\bf W}$, we let $g_w \in \bf N$ denote the corresponding element.

\begin{lem}
The map $\Psi:\cal G\to \bf N^{\bf W}$ given by $\Psi(G)(w)=g_w$ is a continuous map.
\end{lem}

\begin{proof}
It is enough to show, for any $w\in {\bf W}$ and $m\in \bf N$, that the set $$\cal G_{w,m}:=\{G\in \cal G \ : \ g_w=m\}$$ is open in $\cal G$, which we prove by recursion on the length of $w$.  This is obvious when $w$ is a variable.  When $w$ is the inverse of a variable, say $x^{-1}$, then $\cal G_{w,m}=\{G \ : \ \iota_\bd G(a)=m\}$, which is clearly open.  Now suppose that $w=w_1\cdot x_i$.  Then $$\cal G_{w,m}=\bigcup_{n\in \bf N}\{G\in \cal G \ : \ w_1^{\bd G}=n\text{ and }\mu_\bd G(n,a_i)=m\}$$ which is open by the induction hypothesis.  Similarly, if $w=w_1\cdot x_i^{-1}$, then $$\cal G_{w,m}=\bigcup_{n,p}\{ G\in \cal G \ : \ w_1^{G}=n\text{ and }\iota_\bd G(a_i)=p \text{ and }\mu_\bd G(n,p)=m\}$$ which is again open.
\end{proof}


For any system $\Sigma(\vec x)$ and $\vec a\in \bf N^{n_\Sigma}$, set $[\Sigma(\vec a)]=\{G \in \cal G \mid G\models \Sigma(\vec a)]$.

\begin{cor}\label{cor: top}
The sets $[\Sigma(\vec a)]$, as $\Sigma$ ranges over all systems and $\vec a$ ranges over $\bf N^{n_\Sigma}$, form a basis for $\cal G$ consisting of clopen sets.
\end{cor}

\begin{proof}
For any word $w(\vec x)$ and any $\vec a\in \bf N$, we see that $$[w(\vec a)=e]=\bigcup_{n\in \bf N}\{\bd G\in \cal G \ : \ e_G=n \text{ and }\Psi(\bd G)(w)=n\},$$ which is open by the continuity of $\Psi$.  On the other hand, $$[w(\vec a)=e]=\bigcap_{n\in \bf N}\{\bd G\in \cal G \ : \ e_G\not=n \text{ or }\Psi(\bd G)(w)=n\},$$ which is closed by the continuity of $\Psi$ as well.  It follows that $[w(\vec a)=e]$ is a clopen subset of $\cal G$.  It follows immediately that every set of the form $[\Sigma(\vec a)]$ is also clopen.  The union of these sets clearly cover $\cal G$:  given $\bd G\in \cal G$, if $\iota_\bd G(1)=n$, then $G\in [1\cdot n=e]$.  It is easy to see that these sets are closed under finite intersections. Moreover, for each open set in $\mathcal{G}$, one can find a family of sets of the form $[\Sigma(\vec a)]$ (using the multiplication table) whose union is the given open set. Therefore, these clopen sets form a basis.
\end{proof}

The following is obvious but worth recording:

\begin{prop}
Any permutation $\sigma$ of $\bf N$ induces a homeomorphism $\sigma^\#$ of $\cal G$ for which $\sigma^\#[\Sigma(\vec a)]=[\Sigma(\sigma(\vec a))]$.
\end{prop}

Given $\vec a=(a_1,...,a_n)\subset \mathbf{N}^n$ for some $n\in \mathbf{N}$ and an enumerated group $G\in \mathcal{G}$,
we denote by $\langle \vec a\rangle_G$ as the subgroup generated by the elements $a_1,...,a_n$ in $G$.

Let $P$ be a property of countable groups that is closed under direct sums for which $\mathcal{G}_P$ is Polish. 
Consider a system of equations and inequations $\Sigma(\vec x)$. Note that $[\Sigma(\vec a)]\cap \mathcal{G}_P$ is nonempty for some $\vec a\in \bf N^{n_\Sigma}$ if and only if it is nonempty for all $\vec a\in \bf N^{n_\Sigma}$. If this is the case, we call $\Sigma(\vec x)$ a \textbf{$P$-system}. For $\Sigma(\vec x)$  a $P$-system, we define the sets
$$[\Sigma(\vec a)]_P=[\Sigma(\vec a)]\cap \mathcal{G}_P$$
and
$$\mathcal{X}_{\Sigma,P}:=\bigcup_{\vec a\in \bf N^{n_{\Sigma}}}[\Sigma(\vec a)]_P$$

The proof of the following is similar to that of Corollary \ref{cor: top}.
\begin{lem}
 The sets $[\Sigma(\vec a)]_P$, where $\Sigma$ is a $P$-system and $\vec a\in \mathbf{N}^{n_{\Sigma}}$, form a basis of clopen sets for the induced subspace topology on $\mathcal{G}_P$.
 \end{lem}
The following is a fundamental observation concerning the sets $\mathcal{X}_{\Sigma,P}$.

\begin{lem}\label{closed under products density}
If $\Sigma(\vec x)$ is a $P$-system then the set $\mathcal{X}_{\Sigma,P}$ is an open dense subset of $\mathcal{G}_P$.
\end{lem}

\begin{proof}
The set $\mathcal{X}_{\Sigma,P}$ is an open subset of $\mathcal{G}_P$ by definition.  To see that it is dense, fix a nonempty basic open set $[\Delta(\vec a)]_P$, for a $P$-system $\Delta(\vec x)$ and $\vec a\in \mathbf{N}^{n_{\Delta}}$. Choose $G\in [\Delta(\vec a)]_P$.  Take $\vec b\in \bf N^{n_\Sigma}$ disjoint from $\vec a$ and fix $H\in [\Sigma(\vec b)]_P$.  Let $K$ denote an enumeration of the isomorphism type of $G\oplus H$ such that $$K\models\Sigma(\vec b)\qquad K\models \Delta(\vec a)$$
Since the property $P$ is closed under direct sums, it follows that $K\in [\Delta(\vec a)]_P\cap \mathcal{X}_{\Sigma,P}$.
\end{proof}

Given a first-order theory $T$ of groups, we let $\frak C_T$ denote the class of countable models of $T$ and we let $\cal C_T:=\rho^{-1}(\frak C_T)$. 
\begin{prop}\label{closeduniversal}
Suppose that $\cal C$ is a saturated subclass of $\cal G$ such that $\frak C$ is closed under subgroups.  Then the following are equivalent:
\begin{enumerate}
      \item $\cal C$ is closed in $\cal G$.
    \item $\cal C=\cal C_T$ for some \emph{universal} theory $T$ extending the theory of groups.
\end{enumerate}
\end{prop}

\begin{proof}
First suppose that $\cal C$ is a closed subset of $\cal G$ and set $$T=\{\sigma \ : \ G\models \sigma \text{ for all }G\in \frak C\}.$$  Suppose that $G$ is a countable group satisfying $G\models T$.  We show that $G\in \frak C$.  Towards this end, fix an enumeration $\bd G$ of $G$ with $e_\bd G=1$.  For each $n\in \bf N$, let $\Sigma_n(\vec x)$ be the system of equations that determines the products $\mu_\bd G(i,j)$ and inverses $\iota_\bd G(i)$ for $1\leq i,j\leq n$.  Since $\Sigma_n$ has a solution in $G$, it must have a solution in some group $G_n\in \frak G$, for otherwise $\forall \vec x\bigvee_{\varphi(\vec x)\in\Sigma(\vec x)}\neg\varphi(\vec x)$ belongs to $T$, contradicting that $G\models T$.  

Let $\bd G_n$ be an enumeration of $G_n$ so that $e_{\bd G_n}=1$ and so that, for every $1\leq i,j\leq n$, we have
$$\bd G_n\in [i\cdot j=\mu_\bd G(i,j)]\cap [i^{-1}=\iota_\bd G(i)].$$ It follows that $\lim_{n\to \infty}\bd G_n=\bd G$.  Since $\cal G$ is closed, we have that $\bd G\in \cal C$, whence $G\in \frak C$, as desired.  Consequently, $\frak G=\frak G_T$.  Since $\frak G$ is closed under subgroups, it follows that $T$ is universal.

Now suppose that $\frak G=\frak G_T$ with $T$ a universal theory.  Suppose also that $\bd G_n$ is a sequence from $\cal C$ with $\lim_{n\to \infty}\bd G_n=\bd G$.  We must show that $\bd G\in \cal C$.  To see this, fix a universal axiom $\sigma$ of $T$; it suffices to show that $G\models \sigma$.  Write $\sigma=\forall \vec x \varphi(\vec x)$, where $\varphi(\vec x)=\Sigma_1(\vec x)\vee\cdots\vee \Sigma_m(\vec x)$, a finite disjuntion of systems.  Suppose, towards a contradiction, that there is $\vec a\in \bf N^{n_\varphi}$ so that $\bd G\not\models \varphi(\vec a)$.  As a result, for each $i=1,\ldots,m$, there is an equation $w_i(\vec x)=e$ and $\epsilon_i\in \{0,1\}$ such that $$(w_i(\vec x)=e)^{\epsilon_i}\in \Sigma_i(\vec x)\qquad \text{ but }\bd G\models (w_i(\vec a)=e)^{1-\epsilon_i}$$ (Here by $(w_i(\vec x)=e)^{\epsilon_i}$ we denote $w_i(\vec x)=e$ if $\epsilon_i=1$ and $w_i(\vec x)\neq e$ if $\epsilon_i=0$.)

Let $$\Sigma(\vec x)=\{(w_i(\vec x)=1)^{1-\epsilon_i} \ : \ i=1,\ldots,m\}$$  Since $\bd G\in [\Sigma(\vec a)]$ and $\bd G_n\to \bd G$, there is $n\in \bf N$ such that $\bd G_n\in [\Sigma(\vec a)]$.  Since $G_n\models \varphi(\vec a)$, this is a contradiction.
\end{proof}




\subsection{The Grigorchuk space of marked groups}\label{marked}

In geometric group theory, a different topological space is often used when studying the space of all finitely generated groups, namely the \textbf{Grigorchuk space of marked groups} $\cal M$.  
This space was
first systematically studied by Grigorchuk in \cite{GrigorchukGrowth}, considered by Gromov
in \cite{Gromov} in his celebrated work on groups of polynomial growth, and has an antecedent in the
Chabauty topology (see \cite{BridsonDeLaHarpe}).
Recall that a \textbf{marked group} is a pair $(G,S)$, where $G$ is a group and $S=\{s_1,...,s_n\}$ is a finite, ordered generating set for $G$.
Such pairs are considered up to equivalence by \textbf{marked isomorphisms}, that is, $(G_1,S_1),(G_2,S_2)$ are equivalent if $S_1$ and $S_2$ have the same length and the unique order-preserving bijection between $S_1$ and $S_2$ extends to an isomorphism between $G_1$ and $G_2$.

The collection of all marked groups whose marking has size $n$ is denoted by $\mathcal{M}_n$
and is endowed with the topology induced by the following pseudometric:
two marked groups $(G_1,S_1)$ and $(G_2,S_2)$ are distance $e^{-n}$ apart if $n$ is the largest number such that the $n$-balls around the identity in the respective Cayley graphs admit graph isomorphisms that preserve the order of the markings (emerging as edge labels).  Note that two marked groups are at distance $0$ from either other precisely when they are equivalent in the sense of the previous paragraph, whence this pseudometric on the collection of all marked groups descends to an actual metric on the set of equivalence classes.

Another way to provide the same topology on $\mathcal{M}_n$ is as follows.
Consider a marked group $(G,S)$, where $S=\{s_1,...,s_n\}$.
Let $\mathbf{F}_n$ be the free group of rank $n$ which is freely generated by $f_1,...,f_n$.
We identify $(G,S)$ with $Ker(\phi)\in \{0,1\}^{\mathbf{F}_n}$, where $\phi:\mathbf{F}_n\to G$
is the homomorphism determined by mapping $f_i\mapsto s_i$ for all $1\leq i\leq n$.
In this way, we may view $\mathcal{M}_n$ as a subset of $\{0,1\}^{\mathbf{F}_n}$.
It is easy to see that, after this identification, $\mathcal{M}_n$ is a closed subspace of $\{0,1\}^{\mathbf{F}_n}$ and the aforementioned topology on $\mathcal{M}_n$ is the same as the subspace topology inherited from $\{0,1\}^{\mathbf{F}_n}$. 
This perspective makes it clear that $\mathcal{M}_n$ is a totally disconnected, compact, Haursdorff topological space.

The map $$(G,\{x_1,...,x_m\})\xhookrightarrow{} (G,\{x_1,...,x_m,id_G\})$$ induces a natural inclusion $$\mathcal{M}_n\xhookrightarrow{} \mathcal{M}_{n+1}.$$
The directed union $$\mathcal{M}=\bigcup_{i\in \mathbf{N}\setminus \{0\}} \mathcal{M}_n$$ is called the \textbf{space of marked groups}.

A related space of marked groups, which also accommodates infinitely generated groups, is the following.
Let $\bf F_\infty$ denote the free group on the generators $\{x_i \ : \ i\in \bf N\}$.  Then the set of all normal subgroups of $\bf F_\infty$ is a closed subset of $\mathcal P(\bf F_\infty)$ when this latter space is identified with the compact space $2^{\bf F\infty}$.  To each normal subgroup $N$ of $\bf F_\infty$, one obtains the countable \textbf{marked group} $\bf F_\infty/N$.  Clearly every countable group can be marked in this way and consequently the compact space $\cal M_{\infty}$ of marked groups serves as another topological space for dealing with all countable groups.  Notice that this method also allows for one to deal with finite groups.\footnote{The space of enumerated groups could be adapted to accommodate finite groups as well, but since finite groups are uninteresting for our purposes, we chose to deal with the simpler set-up above.}  Note that $\cal G$ is not compact, whence $\cal G$ and $\cal M_{\infty}$ are not homeomorphic; in other words, these topological models for dealing with countable groups are genuinely different.  Nevertheless, we do have:

\begin{prop}\label{compare}
The map $\tau:\cal G\to \cal M_{\infty}$ given by $$\tau(\bd G):=\{w(x_1,\ldots,x_n)\in \bf F_\infty \ : \ w(1,\ldots,n)^\bd G=e\}$$ is a continuous surjection.
\end{prop}

\begin{proof}
It is clear that $\tau$ is continuous.  To see that it is open, it suffices to see that the preimages of the subbasic open sets $\{N\in \cal M \ : \ w(x_1,\ldots,x_n)\in N\}$ and $\{N\in \cal M \ : \ w(x_1,\ldots,x_n)\not\in N\}$ are open in $\cal G$.  However, these preimages are simply $[w(1,\ldots,n)=e]$ and $[w(1,\ldots,n)\not=e]$ respectively, which are both open in $\cal G$. 
\end{proof}

\begin{remark}
As pointed out to us by Alekos Kechris, although the space of enumerated groups and $\cal M_{\infty}$ are not homeomorphic, they induce the same \emph{Borel structure} on the set of isomorphism classes of countable groups.  More precisely, one can equip $\frak G$ with the largest $\sigma$-algebra $\cal B_\rho$ for which the map $\rho$ is measurable (where $\cal G$ is equipped with its Borel $\sigma$-algebra).  If one lets $\rho':\cal M\to \frak G$ denote the analogous reduction map, then the corresponding $\sigma$-algebra $\cal B_{\rho'}$ coincides with $\cal B_\rho$.  In other words, one can, in a Borel manner, recover an enumeration of a given countable group from a marking of that group and vice-versa.
\end{remark}

We can recover the spaces $\cal M_n$ from the space $\cal M_\infty$ by noting that, for each $n\in \bf N$, we have that $\cal M_n$ can be identified (as a topological space) with $\{N\in \cal M_\infty \ : \ i\in N \text{ for all }i>n\}$ (endowed with the subspace topology).  The proof of the following proposition is analogous to the proof of Proposition \ref{compare} above:

\begin{prop}
For each $m$, the map $\tau_m:\cal G\to \cal M_m$ given by $\tau_m(\bd G):=\langle 1,\ldots,m\rangle$ (viewed as a marked group) is a continuous surjection.
\end{prop}

\subsection{Notions of orderability of countable groups}\label{orderability}

Now we recall the notions of orderability that we study in this article.
We often state various well-known facts and definitions and refer the reader to \cite{deroinnavasrivas} for a comprehensive survey on the topic, including the proofs of many of these facts.

\begin{defn}
A group $G$ is \textbf{left orderable} (resp. \textbf{bi-orderable})  if there exists a total order on the group that is invariant under left translation (resp. left and right translation), that is, given any $f,g,h\in G$, if $f<g$ then $hf<hg$ (resp. $hf<hg$ and $fh<gh$).  A left-orderable group equipped with a particular left-invariant order will be called a \textbf{left ordered group}.
\end{defn}


Let $G$ and $H$ be left-ordered groups. A homomorphism $f:G\to H$ is \textbf{monotone increasing} if, for every $g,h\in G$, we have $g<h\implies f(g)\leq f(h)$.

The following fact is well-known:

\begin{lem}\label{orderingextensions}
Consider the short exact sequence of groups $$1\to N\xrightarrow{i} G\xrightarrow{p} Q\to 1.$$
If $N$ and $Q$ are left-ordered groups, then there  exists a unique left-invariant total order on $G$ for which $i$ and $p$ become monotone increasing.
\end{lem}

The following are striking results of Bludov and Glass (see \cite{bludovglass2} and  \cite{bludovglass} respectively.)

\begin{thm}\label{bludovglass}
Let $G_1$ and $G_2$ be left-ordered groups with subgroups $H_1$ and $H_2$, respectively. 
If $\phi:H_1\to H_2$ is an order-preserving isomorphism, then the free product of 
$G_1$ and $G_2$ with $H_1$ and $H_2$ amalgamated via $\phi$ admits a left-invariant order extending the orders on $G_1$ and $G_2$.
\end{thm}

\begin{thm}\label{bludovglass2}
The following holds:
\begin{enumerate}
\item Every recursively presented left orderable group embeds in a finitely presented left orderable group.
\item Left-orderability is a Boone-Higman property, that is, a finitely generated left orderable group has solvable word problem if and only if it can be embedded in a simple left orderable group which can be embedded in a finitely presented left orderable group.
\end{enumerate}
\end{thm}

Our next notion of orderability is presented in the following fact:

\begin{fact}\label{C-orderability}
For any group $G$, the following are equivalent:
\begin{itemize}
    \item $G$ is \textbf{locally indicable}, that is, every nontrivial finitely generated subgroup of $G$ has an infinite cyclic quotient.
    \item $G$ is \textbf{C-orderable}, that is, there is a total order $<$ on $G$ that is left invariant and moreover, for each pair $f,g\in G, f,g>id_G$, it holds that $fg^2>g$.
\end{itemize}

\end{fact}

In the sequel, we prefer to use the terminology ``locally indicable'' rather than ``C-orderable.''

\begin{defn}\label{Upp definition}
A group $G$ is said to satisfy the \textbf{unique product product property (UPP)} if, for all pairs of finite subsets $X$ and $Y$ of $G$, there exists an element $g\in G$ such that:
\begin{itemize}
    \item $g= xy$ for some $x\in X$ and $y\in Y$, but
    \item $g\neq x'y'$ for all $x'\in X\setminus \{x\}$ and  $y'\in Y\setminus\{y\}$.
\end{itemize} 
\end{defn}
\begin{fact}
Every left orderable group satisfies the UPP. 
\end{fact}
\begin{proof}
Let $G$ be a left-ordered group and consider a pair of finite subsets $X$ and $Y$ of $G$. Set $y_0\in Y$ to be the largest element in $Y$. Then, clearly, $xy_0 >xy'$ for all $y'\in Y\setminus\{y_0\}$. Now let $x_0\in X$ be such that $x_0y_0$ is the largest element of the set $\{xy_0\}_{x\in X}$. It is easy to see that $x_0y_0$ satisfies the definition above. 
\end{proof}

The following fact provides characterizations of the aforementioned orderability conditions on a group (see \cite{GOD} for a proof):

\begin{fact}\label{orderabilityopen}

Given a group $G$, we have (see \cite{GOD}):

\begin{enumerate}
\item $G$ is left orderable if and only if, for any finite subset $F=\{f_1,\hdots, f_n\}$ of $G$, there exists $E= (\epsilon_1,\hdots ,\epsilon_n)\in \{1,-1\}^n$ such that $id_G$ does not belong to the semigroup generated by $F^E=\{f_1^{\epsilon_1},\hdots, f_n^{\epsilon_n}\}$.

\item $G$ is locally indicable if and only if, for every finite subset $F=\{f_1,\hdots f_n\}$ of $G\setminus \{id_G\}$, there exists $E= (\epsilon_1,\hdots, \epsilon_n)\in \{1,-1\}^n$ such that the identity is not contained in $\langle\langle F^E\rangle\rangle$, which is the smallest semigroup that satisfies the following conditions: 
\begin{itemize}
     \item  $\{f_1^{\epsilon_1},...f_n^{\epsilon_n}\}\subseteq S$
     
    \item for all $g_1, g_2\in S$, the element $g_1^{-1}g_2g_1^{2}$ lies in $S$. 
 \end{itemize}
 
\item  $G$ is biorderable iff, for every finite subset $F=\{f_1,\hdots f_n\}$ of $G\setminus \{id_G\}$, there exists $\epsilon_1,\hdots, \epsilon_n\in \{1,-1\}$ such that the identity is not contained in the smallest semigroup $S$ that satisfies the following conditions: 
 \begin{itemize}
     \item  $\{f_1^{\epsilon_1},...f_n^{\epsilon_n}\}\subseteq S$
     
    \item for all $g_1, g_2\in S$, the elements $g_1g_2g_1^{-1}$ and $g_1^{-1}g_2g_1$ also lie in $S$. 
 \end{itemize} 
\end{enumerate}
\end{fact}

\begin{lem}\label{Lemma: OrderabilityClosedProperties}
The properties of being left orderable, locally indicable, biorderable , being torsion-free, and having unique products are all closed properties.
\end{lem}

\begin{proof}
It is clear that containing torsion is an open property. It follows from Fact \ref{orderabilityopen} that the negation of left orderability, local indicability and biorderability are open properties. It follows from Definition \ref{Upp definition} that the negation of the unique product property is open.
\end{proof}

\subsection{The relevant subspaces} 
We now introduce the various saturated subspaces of $\cal G$ which will be the focus of this paper.
First, we consider the following:
\begin{itemize}
\item $\Gsm:=\{\bd G\in \cal G \ : \ G \text{ does not contain $\mathbf{F}_2$ subgroups}\}$
\item $\Gam:=\{\bd G\in \cal G \ : \ G \text{ is amenable}\}$
\item $\Gll:=\{\bd G\in \cal G \ : \ G \text{ is lawless}\}$
\item $\Gsmll:=\Gsm\cap \Gll$
\item $\Gamll:=\Gam\cap \Gll$
\item $\cal G_w:=\{\bd G \ : \ G \text{ obeys the law }w=e\}$
\item $\Gamw:=\Gam\cap \cal G_w$
\end{itemize}

In the first item above, the subscript $\operatorname{sm}$ stands for ``small'' as groups not containing a nonabelian free subgroup are often given this name.

\begin{thm}
All seven subspaces from the previous list are $G_\delta$ subspaces of the space $\cal G$, whence Polish.  Moreover, $\cal G_w$ is actually closed.
\end{thm}

\begin{proof}
To see that $\Gsm$ is $G_\delta$, it suffices to notice that $$\Gsm=\bigcap_{a,b\in \bf N}\bigcup_{w(a,b)}[w(a,b)=e],$$ where the union ranges over all nontrivial words $w$.  

In order to show that $\Gam$ is $G_\delta$, we remind the reader that a group $G$ is amenable if and only if it satisfies the \textbf{Folner condition}.  More precisely, given a finite set $F\subseteq G$ and $\epsilon>0$, a nonempty finite set $K\subseteq G$ is called a $(F,\epsilon)$-Folner set if, for each $g\in F$, we have $|gK\triangle K|<\epsilon |K|$.  We then have that $G$ is amenable if and only if, for every finite $F\subseteq G$ and $\epsilon>0$, there is a finite $(F,\epsilon)$-Folner subset of $G$.  For any $\vec a\in \bf N^m$, $\vec b\in \bf N^n$, and $\epsilon>0$, we let $U_{\vec a,\vec b,\epsilon}$ denote the open set
$$\bigcap_{1\leq j<k\leq n}[b_i\not=b_j]\cap \bigcap_{i=1}^m\bigcup_{I\subseteq_\epsilon [n]}\bigcap_{j\in I}\bigcup_{k=1}^n[a_ib_j=b_k],$$ where the notation $I\subseteq_\epsilon [n]$ indicates that $|I|>(1-\epsilon)n$.  We then have that


$$\Gam=\bigcap_{\vec a\in \bf N^{<\bf N}}\bigcap_{\epsilon\in \bf Q^{>0}}\bigcup_{\vec b\in \bf N^{<\bf N}}U_{\vec a,\vec b,\epsilon}.$$
To see that $\Gll$ is $G_\delta$, it suffices to show that $$\Gll=\bigcap_{w}\bigcup_{\vec a\in \bf N^{n_w}}[w(\vec a)\not=e]$$
Finally, we note that $$\cal G_w=\bigcap_{\vec a\in \bf N^{n_w}}[w(\vec a)=e].$$
\end{proof}

Before moving on to our second collection of properties, we first recall the definition of sofic groups in the context of graph approximations. Let $G$ be a finitely generated group with symmetric generating set $S$. Let $\Gamma$ be a finite directed graph such that each directed edge of $\Gamma$ is labeled by an element of $S$. We say that $\Gamma$ is an \textbf{$n$-approximation (for $n\geq 1$) of the Cayley graph $Cay(G,S)$ of $G$ with respect to $S$} if there exists a subset $W\subseteq V(\Gamma)$ such that the following holds: 
\begin{enumerate}
    \item $|W| > \left( 1- \frac{1}{n} \right)|V(\Gamma)|$ and,
    \item if $p\in W$, then the $n$-neighborhood of $p$ is rooted isomorphic to the $n$-neighborhood of a vertex of $Cay(G,S)$ as edge-labeled graphs.
\end{enumerate}

\begin{defn}\label{soficity via graph approximations}
A finitely generated group $G$ is \textbf{sofic} if there is some (equiv. any) finite generating set $S$ of $G$ for which, given any $n\in \mathbf{N}$, $n>1$, there exists an $n$-approximation of $Cay(G,S)$ by a finite graph. More generally, an arbitrary group is sofic if every finitely generated subgroup is sofic (in the sense of the first part of the definition).
\end{defn}

Next, we consider the following saturated spaces of $\cal G$:

\begin{itemize}

\item $\mc{G}_{lo}$: the space of left orderable enumerated groups.
\item $\mc{G}_{bo}$: the space of biorderable enumerated groups.
\item $\mc{G}_{li}$: the space of locally indicable enumerated groups.
\item $\mc{G}_{upp}$: the space of enumerated groups with the unique product property.
\item $\mc{G}_{sofic}$: the space of sofic enumerated groups.
\item $\mc{G}_{tf}$: the space of torsion-free enumerated groups.
\end{itemize}

Let $\cal P_1$ denote the set of properties appearing in the previous list.  The next general result will allow us to conclude that $\cal G_P$ is a closed subspace of $\cal G$ for each $P\in \cal P_1$:

\begin{prop}\label{closed properties}
Let $P$ be a closed property in $\mathcal{M}$ which is closed under subgroups and direct unions.
Then $\mathcal{G}_P$ is a closed subspace of $\cal G$.
\end{prop}

\begin{proof}
Let $G$ be an enumerated group which does not satisfy $P$. Since $G$ is an increasing union of 
the subgroups $\{\langle 1,...,n\rangle_G\ : \ n\in \mathbf{N}\}$, it follows that there is an $n\in \mathbf{N}$
such that $H=\langle 1,...,n\rangle_G$ does not satisfy $P$.
Since the negation of $P$ is an open property, there is an $m\in \mathbf{N}$ such that the $m$-ball centered around the identity of the Cayley graph of the marked group $(H,\{1,...,n\})$ determines an open subset of $\mathcal{M}_n$ which consists entirely of groups that do not satisfy $P$.
This $m$-ball determines a finite system of equations and inequations $\Sigma(\vec x)$ of arity $n$ such that
$G\models \Sigma(1,...,n)$ and such that, for each $K\in [\Sigma(1,...,n)]$, the marked group $(\langle 1,..., n\rangle_K,\{1,...,n\})$ lies in the aforementioned open subset of $\mathcal{M}_n$.
Consequently, $\langle 1,..., n\rangle_K$ does not have property $P$; since the negation of $P$ is closed under taking overgroups,  we have that $K$ does not have $P$. It follows that $G\in [\Sigma(1,...,n)]\subset \cal G\setminus \mathcal{G}_P$, establishing the desired conclusion.
\end{proof}

\begin{prop}\label{left orderable groups form a Gdelta set}
$\mc{G}_{P}$ is a closed subspace of $\mc{G}$, for each $P\in \mathcal{P}_1$, and hence inherits a Polish topology. 
\end{prop}

\begin{proof}
Lemma \ref{Lemma: OrderabilityClosedProperties} asserts that each of these properties is closed in $\cal M$. Furthermore, it is easy to see that they are closed under subgroups and direction unions. Therefore the result follows from a direct application of Proposition \ref{closed properties}.
\end{proof}

The following result, pointed out to us by Denis Osin, will also be relevant:

\begin{prop}
The set $\{G\in \cal G \ : \ G \text{ is simple}\}$ is a $G_\delta$ subspace of $\cal G$.
\end{prop}

\begin{proof}
It was shown in \cite{minasyanosinwitzel} that simplicity can be expressed by a countable conjunction of $\forall\exists$-sentences, essentially expressing that, given any two elements in the group, one is in the normal closure of the subgroup generated by the other.  The result follows by arguing as previously.
\end{proof}

\begin{remark}\label{Remark: Simple}
One can show that when $\cal C$ is saturated and every group in $\frak C$ can be embedded in a simple group in $\frak C$, we have that $\{G\in \cal C \ : \ G \text{ is simple}\}$ is comeager in $\cal C$.  For the class of amenable groups, this was proven in \cite{grigorchukkravchenko}.  As pointed out to us by Osin, the same can be shown for the class of groups without $\mathbf{F}_2$ subgroups using small cancellation theory. We do not include the proof here, but we can point out to the reader that the proof uses techniques from \cite{osinannals}, adapted to this setting. 
\end{remark}

\section{The proofs}

The goal of this section is to prove the main theorems and the corollaries presented in the introduction.  For the convenience of the reader, we restate the results here.\begin{customthm}{\ref{Theorem: main1}}
Let $P\in \mathcal{P}$ be a property. Then the following hold:
\begin{enumerate}
\item There is a comeager set $\mathcal{X}_P\subset \mathcal{G}_P$ such that, for every enumerated group $G\in \mathcal{X}_P$ and every $P$-near isolated group $H$, $G$ contains a subgroup isomorphic to $H$.
\item Let $Q$ be an open property of finitely generated groups for which there is a finitely generated group that satisfies both $Q$ and $P$. Then there is an open dense set $\mathcal{X}_P\subset \mathcal{G}_P$ such that every enumerated group $G\in \mathcal{X}_P$ contains a finitely generated subgroup satisfying $Q$.

\end{enumerate}
\end{customthm}
\begin{proof}
For (1), let $G$ be a $P$-near isolated group. 
Let $(H,S)\in \mathcal{M}_n$ be a finitely presented marked group and let $X
\subset H$ be the finite set that witnesses the definition of $P$-near isolated for $G$.
Let $H=\langle S\mid R\rangle$ be a finite presentation for $H$ and let $\Gamma_S(H)$ be the corresponding Cayley graph.
There is an $m\in \mathbf{N}$ such that the following holds: Let $B_m$ be the ball of radius $m$ centred at the identity in $\Gamma_S(H)$ which contains the set $X$ as a subset as well as a set of loops that represent all the relations in $R$.  
We denote the open set in $\mathcal{M}$ defined by this open ball also as $B_m$.
Note that by definition, any marked group in $B_m$ that satisfies $P$ contains an isomorphic copy of $G$ as a subgroup.
Moreover, there is at least one nontrivial group in $B_m$ satisfying $P$, (a certain quotient of $H$).  The $m$-ball determines a finite system of equations and inequations $\Sigma_G(\vec x)$ such that for any enumerated group $K$,
and any $\vec a\in \mathbf{N}^{n_{\Sigma_G}}$ such that $K\in [\Sigma_G(\vec a)]_P$, the marked group $(\langle \vec a\rangle_K, \vec a)$ belongs to $B_m$, whence the group $\langle \vec a\rangle_K$ contains a subgroup isomorphic to $G$.

Since there is at least one such $K$ satisfying $P$, we have that $\Sigma_G(\vec x)$ is a $P$-system.
Consequently, Lemma \ref{closed under products density} implies that the set 
$$\mathcal{X}_{\Sigma_G,P}=\bigcup_{\vec{a}\in \mathbf{N}^{n_{\Sigma}}}[\Sigma_G(\vec a)]_P$$ is an open dense set in $\mathcal{G}_P$, and each enumerated group in it contains a subgroup isomorphic to $G$.
Since the set $Y_P$ of $P$-near isolated groups is countable (as $P$-near isolated groups are finitely generated subgroups of finitely presented groups), it follows that $$\mathcal{X}=\bigcap_{G\in Y_P}\mathcal{X_P}_{\Sigma_G,P}$$
is the required comeager set.


For (2), let $G$ be a finitely generated group that satisfies both $P$ and $Q$.
Consider a marking $(G,S)$ of $G$ such that $(G,S)$ admits an open neighbourhood in $\mathcal{M}$ consisting entirely of marked groups that satisfy $Q$.
In particular, there is an $m\in \mathbf{N}$ such that the open set $B_n
\subset \mathcal{M}$ determined by the $m$-ball of the Cayley graph of $\Gamma_S(G)$ satisfies that each marked group in $B_n$ satisfies $Q$.
This determines a finite system of equations and inequations $\Sigma(\vec x)$ such that for any enumerated group $H$,
and any $\vec a\in \mathbf{N}^{n_{\Sigma}}$ such that $H\in [\Sigma(\vec a)]$, it follows that $(\langle \vec a\rangle, \vec a)\in B_n$.
Note that $\Sigma$ is a $P$-system, since for any $\vec a\in \mathbf{N}^{n_{\Sigma}}$, there is an enumeration of $G$ such that $G\in [\Sigma(\vec a)]_P$. So by Lemma \ref{closed under products density}, the desired open dense set is $\mathcal{X}_{\Sigma,P}$.
\end{proof}

\begin{customthm}{\ref{Theorem: main2}}
Let $P\in \mathcal{P}$ be a Boone-Higman property. Then there is a comeager set $\mathcal{X}_P\subset \mathcal{G}_P$ such that every group in $\mathcal{X}_P$ contains an isomorphic copy of every finitely generated group satisfying $P$ with solvable word problem.
\end{customthm}

\begin{proof}
Let $G$ be a finitely generated group that satisfies $P$ and has solvable word problem. Since $P$ is a Boone-Higman property, $G$ embeds in a simple subgroup $H_1$ of a finitely presented group $H_2=\langle S\mid R\rangle$ satisfying property $P$.
It suffices to show that there is an open dense set $\mathcal{X}\subset \mathcal{G}_P$ such that, for each $H\in \mathcal{X}$, $G$ embeds in $H$.
Since the class of finitely generated groups with solvable word problem is countable,
we can conclude the statement of the theorem by taking an intersection of all such open dense sets. 

Take $n\in \mathbf{N}$ such that the $n$-ball $B_n$ of the Cayley graph of $\Gamma_S(H_2)$ determines an open set, also denoted as $B_n$, in $\mathcal{M}_{|S|}$ satisfying the following:
\begin{enumerate}

\item For any $K\in B_n$, $K$ is a quotient of $H_2$.

\item $H_1\cap B_n\neq \{id\}$.

\end{enumerate}

Since $H_1$ is simple and contains $G$ as a subgroup, it follows that any nontrivial quotient of $H_2$ whose restriction to $B_n$ is injective contains $G$ as a subgroup.
This determines a finite system of equations and inequations $\Sigma(\vec x)$
such that, for any $\vec a\in \mathbf{N}^{n_{\Sigma}}$ and $H\in [\Sigma(\vec a)]$, it holds that $(\langle \vec a\rangle_H, \vec a)\in B_n$ and thus $G\leq \langle \vec a\rangle_H\leq H$.
Note that, by our hypothesis, $\Sigma$ is a $P$-system, since for any $\vec a\in \mathbf{N}^{n_{\Sigma}}$, there is an enumeration of $H$ which lies in $[\Sigma(\vec a)]_P$. By Lemma \ref{closed under products density}, the desired open dense set is once again $$\mathcal{X}_{\Sigma,P}=\bigcup_{\vec a\in \mathbf{N}^{n_{\Sigma}}}[\Sigma(\vec a)]_P.$$
\end{proof}

We remark that for the special case when $P$ is the tautological property, this result follows from a result of Neumann (see \cite{NEUMANN1973553}, \cite{Neumann1952ANO}), namely that every existentially closed group contains a copy of every finitely generated group with solvable word problem. (The notion of existentially closed groups is introduced in Definition \ref{Definition: Existentially closed}, and it is shown in Lemma \ref{ecgen} that the set of existentially closed groups form a comeager set in $\mathcal{G}$).

\begin{customthm}{\ref{Theorem: main3}}
Let $P\in \mathcal{P}$ be a Boone-Higman property that is inherited by subgroups. Then exactly one of the following holds:
\begin{enumerate}

\item $P$ is not strongly undecidable.  In this case, there is a comeager set $\mathcal{X}_{P}\subset \mathcal{G}_{P}$ such that for each group $G\in \mathcal{X}_{P}$, the set of finitely generated subgroups of $G$ coincides with the set of finitely generated groups satisfying $P$ that also have a solvable word problem.

\item $P$ is strongly undecidable.  In this case, there is a comeager set $\mathcal{X}_P\subset \mathcal{G}_P$ such that, for each group $G\in \mathcal{X}_P$, the set of finitely generated subgroups of $G$ contains all finitely generated groups satisfying $P$ that also have a solvable word problem, but also contains a finitely generated subgroup with $P$ that has an unsolvable word problem.

\end{enumerate}

\end{customthm}
\begin{proof}

We first handle the case when $P$ is not strongly undecidable. 

{\textbf{Claim}}: For each $m\in \mathbf{N}$, there is a comeager subset $\cal Z_m$ of $\cal G_P$ satisfying $$\mathcal{Z}_m
\subseteq\{G\in \mathcal{G}_P\mid \langle 1,...,m\rangle_G\text{ has a solvable word problem}\}.$$

Note that once the claim is proved, we can finish the proof of part $(1)$ by setting $\mathcal{X}_P=\bigcap_{m\in \mathbf{N}}\mathcal{Z}_m$.

{\textbf{Proof of Claim}}:

For a fixed $m\in \mathbf{N}$, we consider an enumeration of all triples 
$$(G_n,\langle S_n\mid R_n\rangle ,\{W_1^{(n)},...,W_{m+1}^{(n)}\})_{n\in \mathbf{N}},$$
where:
\begin{enumerate}
\item $G_n$ is a finitely presented group endowed with a finite presentation $\langle S_n\mid R_n\rangle$ with ordered generating set $S_n=(s_{1,n},...,s_{k_n,n})$.
\item $\{W_1^{(n)},...,W_{m+1}^{(n)}\}$ is a set of words in the generating set $S$ with the property that there is a subgroup $H_n\leq G_n$ that is simple and $\{W_1^{(n)},...,W_{m+1}^{(n)}\}\subset H_n\setminus \{id\}$.
\end{enumerate}

Given $\vec x=(x_1,...,x_{k_n})$ and a relation $R\in R_n$, we let $R(\vec x)$ denote the word obtained by replacing every occurrence of $s_{i,n}^{\pm 1}$ by $x_i^{\pm 1}$. We define $W_1^{(n)}(\vec x),...,W_m^{(n)}(\vec x)$ in a similar fashion.

For each $n\in \mathbf{N}$, let $\Sigma_n(\vec x)$ (where $\vec x=(x_1,...,x_{k_n})$) be a finite system comprising of the following equations and inequations:
$$R(\vec x)=1\qquad \text{ for all }R\in R_n$$ $$W_1^{(n)}(\vec x)=1\qquad W_2^{(n)}(\vec x)=2\qquad ...\qquad W_m^{(n)}(\vec x)=m$$ 
$$ W_{m+1}^{(n)}(\vec x)\neq e.$$

We define the set $$\mathcal{Z}_m=\bigcup_{n\in \mathbf{N}} \mathcal{X}_{\Sigma_n, P}.$$
The set $\mathcal{Z}_m$ is clearly open. Recalling that finitely generated subgroups of simple subgroups of finitely presented groups have a solvable word problem, we claim that our construction ensures that $$\mathcal{Z}_m
\subseteq\{G\in \mathcal{G}_P\mid \langle 1,...,m\rangle_G\text{ has a solvable word problem}\}.$$ 

To see this, consider $G\in \mathcal{Z}_m$ such that $G\models \Sigma_n(\vec a)$ for some $n\in \mathbf{N}$ and $\vec a\in \mathbf{N}^{k_n}$.
It follows from the definition of the system that there is a finitely presented group $H_1$, a simple subgroup $K\leq H_1$, a finite subset $X\subset K\setminus \{id\}$, and a surjective homomorphism 
$\phi: H_1\to H_2$ such that:
\begin{enumerate}
\item $\phi$ is injective on $K$. (In particular, $\phi\restriction \langle X\rangle$ is injective.)
\item There is an isomorphism $\lambda: H_2\to \langle \vec a\rangle_G$ whose restriction induces isomorphisms 
$$\phi(K)\to \lambda(\phi(K))\qquad \text{ and }\qquad \phi(\langle X\rangle)\to  \langle 1,...,m\rangle_G.$$
\end{enumerate}
It follows that $\langle 1,...,m\rangle_G$ embeds in a simple subgroup of a finitely presented group, whence our conclusion follows.

It remains to show that $\mathcal{Z}_m$ is dense in $\mathcal{G}_P$.
Let $\Delta(\vec x)$ be a $P$-system and let $\vec a\in \mathbf{N}^{n_{\Delta}}$.
We would like to show that $[\Delta(\vec a)]_P\cap \mathcal{Z}_m\neq \emptyset.$

Towards that end, fix an arbitrary order on the set $S=\{\vec a\}\cup \{1,...,m\}$.
We consider the finitely presented group $H$ with presentation $\langle S\mid R\rangle$, where 
$S$ is as above and $R$ is determined by the set of equations satisfied by $\vec a$ in $\Delta(\vec a)$.
Let $X$ be the finite subset of $H$ that is determined by the set of inequations satisfied by $\vec a$
in $\Delta(\vec a)$. 

Note that while it may be the case that no enumeration of $H$ is in $[\Delta(\vec a)]_P$, there is at least one quotient $\phi:H\to H_1$, injective on $X$, such that there is an enumeration of $H_1$ in $[\Delta(\vec a)]_P$. (This enumeration of $H_1$ will satisfy that the ordered image $\phi(S)$ is in an order preserving bijection with the order we fixed on $\{\vec a\}\cup \{1,...,m\}$.)
Since $P$ is not strongly undecidable, we may choose an appropriate such quotient (and enumeration) so that $H_1$ has a solvable word problem.  Since $H_1$ has a solvable word problem and since $P$ is a Boone-Higman property, it embeds in a simple subgroup of a finitely presented group $H_2$ satisfying $P$.
Using this, it is easy to see that there is an $n\in \mathbf{N}$ and an enumeration of $H_2$ 
such that $H_2\in \mathcal{X}_{\Sigma_n,P}\cap [\Delta(\vec a)]_P$.

Now we treat the case when $P$ is strongly undecidable.
By definition, there is a finitely presented group $G=\langle S\mid R\rangle$ and a finite subset $X\subset G$ such that:
\begin{enumerate}
\item There is at least one group $H$ satisfying $P$ for which there is a surjective homomorphism $\phi:G\to H$ whose restriction to $X$ in injective.
\item Every surjective homomorphism $\phi:G\to H$ whose restriction to $X$ in injective and for which $H$ satisfies $P$ also satisfies that $H$ has an unsolvable word problem.
\end{enumerate}
Let $n\in \mathbf{N}$ be such that the ball of radius $n$ centred at the identity of the Cayley graph of $\Gamma_S(G)$ contains the set $X$ as a subset and also loops that witness the finite set of relations $R$. We denote the open set in $\mathcal{M}_{|S|}$ defined by this open ball by $B_n$.

This $n$-ball determines a finite system of equations and inequations $\Sigma_G(\vec x)$ such that, for any enumerated group $K$
and any $\vec a\in \mathbf{N}^{n_{\Sigma_G}}$ such that $K\in [\Sigma_G(\vec a)]_P$, the marked group $(\langle \vec a\rangle_K, \vec a)$ belongs to $B_n$. By our assumption, it follows that $[\Sigma_G(\vec a)]_P$ is nonempty for any distinct $\vec a\in \mathbf{N}^{n_{\Sigma_G}}$, whence $\Sigma_G(\vec x)$ is a $P$-system. Moreover, for each $H\in [\Sigma_G(\vec a)]_P$, $\langle \vec a\rangle_H$ has an unsolvable word problem.
By Lemma \ref{closed under products density}, the set 
$$\mathcal{X}_{\Sigma_G,P}=\bigcup_{\vec{a}\in \mathbf{N}^{n_{\Sigma}}}[\Sigma_G(\vec a)]_P$$ is an open dense set in $\mathcal{G}_P$, and each enumerated group in it contains a subgroup with an unsolvable word problem, yielding the desired conclusion.
\end{proof}

The proof of Theorem \ref{main theorem 4} will appear in Section 5 below.  We continue with the next theorem from the introduction:

\begin{customthm}{\ref{Theorem ConjOrd}}
Let $P\in \mathcal{P}$ be a property of torsion-free groups that is closed under amalgamation along infinite cyclic subgroups and HNN extensions with associated subgroups that are infinite cyclic. 
Then there is a comeager set $\mathcal{X}_P\subset \mathcal{G}_P$ such that each $G\in \mathcal{X}_P$ has only one nontrivial conjugacy class (in particular, it is simple) and is verbally complete.
\end{customthm}

\begin{proof}

First we shall construct a comeager set $\mathcal{X}\subset \mathcal{G}_P$ 
such that each $G\in \mathcal{X}$ has only one nontrivial conjugacy class.
Consider the system $\Delta_1(\vec x)$ for $\vec x=(x_1,x_2,x_3)$ defined by $x_1^{-1} x_2 x_1=x_3$ and $\Delta_2(y)$ defined by $y=e$.
We claim that for each pair $i,j\in \mathbf{N}, i\neq j$, $$\mathcal{X}_{i,j}=(\bigcup_{k\in \mathbf{N}}[\Delta_1(i,j,k)]_P)\cup [\Delta_2(i)]_P\cup [\Delta_2(j)]_P$$
is an open dense set.

Fix a basic clopen set in the topology on $\mathcal{G}_P$ given by $[\Omega(\vec b)]_P$ for some fixed $\vec b\in \mathbf{N}^{n_{\Omega}}$, where $\Omega(\vec z)$ is a $P$-system. 
Let $H\in [\Omega(\vec b)]_P$.
If either $i=e$ or $j=e$ holds in $H$, then $H\in \mathcal{X}_{i,j}\cap [\Omega(\vec b)]_P$, and we are done.
Assume that this is not the case.
Let $\vec c=(i,j,\vec b)$.
Since $H$ is torsion-free (groups satisfying the property $P$ in the hypothesis of the theorem are always torsion-free),
we construct an HNN extension $K=\langle H, t\mid t^{-1} gt=h\rangle$ where $g,h$ are the elements in $H$ that correspond to the elements $i,j$ in the given enumeration.
By our assumption, $K$ satisfies property $P$, wence we can find an enumeration of $K$ in $\mathcal{G}_P$ such that $\langle \vec c\rangle_K=\langle \vec c\rangle_H$.
 Note in particular that $K\models \Omega(\vec b)$ and hence it follows that 
$K\in \mathcal{X}_{i,j}\cap [\Omega(\vec b)]_P$.
This proves the claim.  The required comeager set is then $$\mathcal{X}=\bigcap_{i,j\in \mathbf{N}, i\neq j}\mathcal{X}_{i,j}.$$

We finish the proof of the theorem by proving that every group in a certain comeager subset of $\mathcal{G}_P$ is verbally complete.
Let $W(x_1,...,x_n)$ be a nontrivial reduced word in the letters $x_1^{\pm},...,x_n^{\pm}$.
For $\vec x=(x_1,...,x_n,x_{n+1})$, consider the system 
$\Sigma(\vec x)$ defined by $$W(x_1,...,x_n)=x_{n+1}.$$
Let $$\mathcal{X}_{W,i}=
\bigcup_{\vec a\in \mathbf{N}^n}\{G\in \mathcal{G}_P\mid G\models \Sigma(\vec b)\text{ for }
\vec b=(\vec a, i)\}=\bigcup_{\vec a\in \mathbf{N}^n}[\Sigma(\vec a,i)]_P.$$
Clearly, this is an open set.
We claim that it is dense in $\mathcal{G}_P$.
Consider a $P$-system $\Omega(\vec z)$ and the corresponding nonempty open set $[\Omega(\vec c)]_P$ for some fixed 
$\vec c\in \mathbf{N}^{n_{\Omega}}$.
Fix $G\in [\Omega(\vec c)]_P$.
If $i=e$ in $G$, then $$G\in  [\Omega(\vec c)]_P\cap \mathcal{X}_{W,i}$$ 
for trivial reasons and we are done. We may thus assume otherwise.  Since $G$ is torsion free, we can construct the amalgamated free product $$H=G*_{i=W}\mathbf{F}_n,$$ where 
$W=W(x_1,...,x_n)\in \mathbf{F}_n$, with $\mathbf{F}_n$ being the free group of rank $n$ freely generated by $x_1,...,x_n$.
We conclude by finding an enumeration of $H$ for which
$\langle \vec c\rangle_H=\langle \vec c\rangle_G$ and hence 
$H\in  [\Omega(\vec c)]_P\cap \mathcal{X}_{W,i}$, as desired.
\end{proof}

Before proving Theorem \ref{locunivcomeager}, we need the following lemma:

\begin{lem}\label{locunivequiv}
For any $H\in \mathcal{G}_P$, the following are equivalent:
\begin{enumerate}
    \item The isomorphism type of $H$ is locally universal for $\mathcal{G}_P$.
    \item For any $P$-system $\Sigma(\vec x)$, $H\in \mathcal{X}_{\Sigma,P}$.
\end{enumerate}
\end{lem}

\begin{proof}
We begin the proof with a

{\textbf{Claim}}: Suppose $G,H$ are countable groups and $\u$ is a nonprincipal ultrafilter on $\bf N$.  Then $G$ embeds into $H^\u$ if and only if any system with a solution in $G$ also has a solution in $H$. 

{\textbf{Proof of claim}}: First suppose that $\alpha:G\hookrightarrow H^\u$ is an embedding and $\Sigma(\vec x)$ is a system with a solution $\vec a\in G$.  Since $\vec h=\alpha(\vec a)\in H^\u$ is a solution of $\Sigma$ in $H^{\u}$, it follows that $\vec h(i)\in H$ is a solution to $\Sigma$ for $\u$-almost all $i\in \mathbf{N}$. 

We now prove the converse.  Enumerate $G=\{g_n \ : \ n\in \bf N\}$ and for each $m\in \mathbf{N}$, let $\Sigma_m(\vec x)$ denote the system $$\{x_j\cdot x_k=x_l \ : \ 1\leq j,k\leq m, \ g_j\cdot g_k=g_l\}.$$  By assumption, $\Sigma_m(\vec x)$ has a solution $(h_{1}^{(m)},...,h_{m}^{(m)})\in \mathbf{N}^{m}$ in $H$.  It follows that the map $G\to H^\u$ given by
$g_n\mapsto f_n$, where $f_n$ is defined by
$$f_n(i)= \left\{\begin{array}{lr}
        e_H & \text{for } 0\leq i< n\\
        h_n^{(i)}, & \text{for } n\geq i\\
        \end{array}\right\}$$

is an injective group homomorphism. This finishes the proof of the claim.

Now we show that the claim implies the conclusion of the lemma.
If the isomorphism type of $H$ is locally universal for $\mathcal{G}_P$, then by definition, any $G\in \mathcal{G}_P$ embeds in $H^{\u}$. Since this holds for any $G\in \mathcal{G}_P$, part $(2)$ follows from the claim. Similarly, if part $(2)$ holds for some $H\in \mathcal{G}_P$, then $(1)$ follows from the claim since it implies that any group $G\in \mathcal{G}_P$ embeds in an ultrapower of $H$.
\end{proof}

\begin{customthm}{\ref{locunivcomeager}}
For each $P\in \mathcal{P}$, $\mathcal{G}_{lu,P}$ is a comeager subset of $\mathcal{G}_P$.
\end{customthm}

\begin{proof}
Using part $(2)$ of the characterisation of locally universal groups in $\mathcal{G}_P$ given by Lemma \ref{locunivequiv} together with Lemma \ref{closed under products density}, it follows that
$$\mathcal{G}_{lu,P}=\bigcap_{\Sigma(\vec x)\text{ a $P$-system}}\mathcal{X}_{\Sigma,P}$$ 
is a comeager set.
\end{proof}

The following is a special case of the \textbf{Downward L\"owenheim-Skolem theorem} and will be used in the proof of Theorem \ref{firstorder}.

\begin{fact}\label{Fact: ElemEquiv}
Given any group $G$ and an infinite subset $X\subseteq G$, there is a subgroup $H$ of $G$ such that $X\subseteq H$, $|H|=|X|$, and so that $G$ and $H$ are elementarily equivalent.
\end{fact} 

The proof of Theorem \ref{firstorder} will also require the following lemma and proposition.

\begin{lem}\label{ultraargument}
Suppose that $H$ is finitely presented and embeds into an ultrapower of $G$.  Then $H$ is fully residually $G$.
\end{lem}

\begin{proof}
Suppose $H=\langle a_1,\ldots,a_m \ | \ w_1,\ldots,w_n\rangle$ and take words $w_1',\ldots,w_p'$ such that $w'_i(\vec a)\not=id$ for all $i=1,\ldots,p$.  Then the system $$\Sigma(\vec x):=\bigwedge_{i=1}^n w_i(\vec x)=id\wedge \bigwedge_{j=1}^p w_j'(\vec x)\not=id$$ has a solution in $H$, whence it also has a solution in $G$, say $\vec b=b_1,\ldots,b_n$.  It follows that the map $a_i\mapsto b_i$ yields a group homormorphism $f:H\to G$ for which $f(w_i'(\vec a))\not=id$ for $i=1,\ldots,p$, as desired.
\end{proof}

\begin{prop}\label{Tprop}
Suppose that $G$ is residually amenable, $H$ is a finitely presented group with property (T), and $G\equiv H$.  Then $G$ and $H$ are both residually finite.
\end{prop}

\begin{proof}
Since $H$ embeds into an ultrapower of $G$, by Lemma \ref{ultraargument} it follows that $H$ is fully residually $G$, so residually amenable, and thus residually finite since $H$ has property (T).  Since $G$ embeds into an ultrapower of $H$, $G$ is residually $H$, and thus also residually finite. 
\end{proof}

\begin{remark}
The proof of the previous proposition shows that one does not need the full strength of the assumption that $G\equiv H$ but rather that each embeds into the ultrapower of the other, or rather, that they have the same universal theory.
\end{remark}

\begin{customthm}{\ref{firstorder}}
There is a comeager set $\mathcal{X}\subset \mathcal{G}_{am}$ such that, for each $G\in \mathcal{X}$, the following holds:
\begin{enumerate}
\item There are continuum many pairwise nonisomorphic countable nonamenable groups with the same first-order theory as $G$.
\item $G$ cannot have the same first-order theory as a finitely presented group with property (T).
\end{enumerate}
\end{customthm}





\begin{proof}
(1)  By \cite{juschenkomonod}, for each $r\in \mathbf R$, there is a group $K_r$ such that the set $\{K_r\ \ : r\in \mathbf{R}\}$ contains continuum many nonisomorphic finitely generated, infinite, simple, amenable groups.  In particular, 
the set of groups $\{\mathbf{F_2}\times K_r\ : \  r\in \mathbf{R}\}$
 contains continuum many pairwise nonisomorphic  finitely generated nonamenable groups.  

Fix a nonprincipal ultrafilter $\u$ on $\bf N$. For any locally universal amenable group G,  since $G$ is lawless, we have that $\mathbf{F}_2\times K_r$ embeds into $G^\u\times G^\u$ (see Fact \ref{lawlessfact} below), which in turn embeds into $(G\times G)^\u$.  Since $G$ is a locally universal element of $\mathcal{G}_{am}$, $(G\times G)^\u$ in turn embeds into $G^\u$.  In summary:  each $\mathbf{F}_2\times K_r$ embeds into $G^\u$.  Using fact \ref{Fact: ElemEquiv}, for each $r\in \bf R$, let $H_r\preceq G^\u$ be a countable subgroup containing $\mathbf{F}_2\times K_r$ which is elementarily equivalent to $G^\u$ (and hence to $G$).  
It follows that the class $\{H_r\ : \  r\in \mathbf{R}\}$ has continuum many isomorphism types and each group in it has the same first order theory as $G$.    

(2)  By Theorem \ref{locunivcomeager} and Proposition \ref{Tprop}, it suffices to show that no group locally universal for $\mathcal{G}_{am}$ can be residually finite.  Recall that Grigorchuk's construction $\ggr$ (of a finitely presented amenable group which is not elementary amenable), is in fact not residually elementary amenable.
It was shown in \cite{grigorchuk98} that every proper quotient of $G_{GR}$ is metabelian. It follows that $G_{GR}$ is an isolated group. Combining this with Theorem \ref{Theorem: main1}, we conclude it embeds in every group in a comeager subset of $\mathcal{G}_{am}$. Using the Baire category theorem, we can assume that every group in this comeager set is locally universal for $\mathcal{G}_{am}$.
This is a contradiction.
\end{proof}

\subsection{Proofs of the applications}
We now provide the proofs of the applications of our main results, as outlined in the introduction.

\begin{customcor}{\ref{Corollary VNDay}}
The following holds:
\begin{enumerate}
\item The generic enumerated group without $\mathbf{F}_2$ subgroups is nonamenable.
\item The generic left orderable enumerated group without $\mathbf{F}_2$ subgroups is nonamenable. 
\item The generic enumerated amenable group is not elementary amenable.
\end{enumerate}
In fact, we can choose these comeager sets to be open dense sets.
\end{customcor}

\begin{proof}
First, we observe the following fact:
 if $G$ is a finitely presented nonsolvable group for which there is an $n\in \mathbf{N}$ such that every proper quotient of $G$ is solvable of length $n$,
then $G$ is isolated.
This follows from the observation that for a fixed finite presentation $\langle S\mid R\rangle$ of $G$ and a sufficiently large $m\in \mathbf{N}$, the $m$-ball centered at the identity of the corresponding Cayley graph $\Gamma_S(G)$ satisfies the following:
\begin{enumerate}
\item It contains loops that represent all the relations in $R$.
\item It contains a nontrivial element of the $n$'th derived subgroup of $G$.
\end{enumerate}
This $m$-ball provides an open subset in $\mathcal{M}_{|S|}$ that witnesses that $G$ is an isolated point.

In \cite{lodhamoore}, the third author with Moore constructed a finitely presented nonamenable, left orderable group without free subgroups, denoted $G_0$.
It was shown in \cite{burilllodhareeves} that $G_0'$ is simple and that every proper quotient of $G_0$ is abelian. It follows that $G_0$ is an isolated group. Combining this with Theorem \ref{Theorem: main1}, we conclude the first two parts of the Corollary.

As mentioned above, in \cite{grigorchuk98}, Grigorchuk constructed the first example of a finitely presented amenable group $G_{GR}$ which is not elementary amenable.
It was shown in \cite{grigorchuk98} that every proper quotient of $G_{GR}$ is metabelian, whence it follows that $G_{GR}$ is an isolated group. Combining this with Theorem \ref{Theorem: main1}, we conclude the last part of the Corollary.
\end{proof}




Before moving on, let us mention that Proposition \ref{remark VNDay} will be proven in the next section.

For the next set of proofs, we recall the following well known examples.
Thompson's group $T$ is the group of piecewise linear orientation-preserving homeomorphisms of the circle $\mathbf{S}^1=\mathbf{R}/\mathbf{Z}$ such that:
\begin{enumerate}
\item Each linear part is of the form $2^n+d$, where $n\in \mathbf{Z}$ and $d\in \mathbf{Z}[\frac{1}{2}]/\mathbf{Z}$.
\item There are only finitely many points where the slopes do not exist and these points lie in $\mathbf{Z}[\frac{1}{2}]$.
\end{enumerate} 

The group $\overline{T}<\textup{Homeo}^+(\mathbf{R})$ is the "lift" of this action to the real line.
In particular, there is a short exact sequence $$1\to \mathbf{Z}\to \overline{T}\to T\to 1$$
where the group $\mathbf{Z}$ is the group of integer translations of the real line and coincides with the center of $\overline{T}$.
Since $T$ is finitely presented, it follows that $\overline{T}$ is also finitely presented.
The group $\overline{T}$ was first studied by Ghys and Sergiescu in \cite{GhysSergiescu} and it has several remarkable features.
This group shall play an important role in the next proof.

Next, we recall that, for a free subgroup $\mathbf{F}_2$ of $SL(2,\mathbf{Z})$,
acting linearly on $\mathbf{Z}^2$, the resulting semidirect product 
$\mathbf{F}_2\ltimes \mathbf{Z}^2$ is locally indicable (and therefore also left orderable) and the pair $(\mathbf{F}_2\ltimes \mathbf{Z}^2,\mathbf{Z}^2)$ has relative property $(T)$.  (We refer to \cite{NavasSmooth} for details, and to \cite{bekkadelaharpevalette2008} for the definition of relative property $(T)$). It follows that $\mathbf{F}_2\ltimes \mathbf{Z}^2$ does not have the Haagerup property (see \cite{bekkadelaharpevalette2008}). Note that this group is finitely presented.

\begin{customcor}{\ref{Corollary LeftOrderable}}
There is a comeager set $\mathcal{X}_{lo}\subset \mathcal{G}_{lo}$ such that each $G\in \mathcal{X}_{lo}$ satisfies:
\begin{enumerate}
\item It is not locally indicable.
\item It does not have the Haagerup property. 
\item It does not admit nontrivial actions by $C^1$-diffeomorphisms on the closed interval or the circle.
\item Contains an isomorphic copy of every finitely generated left orderable group with a solvable word problem.
\end{enumerate}
\end{customcor}

\begin{customcor}{\ref{Corollary LocallyIndicable}}
There is a comeager set $\mathcal{X}_{li}\subset \mathcal{G}_{li}$ such that each $G\in \mathcal{X}_{li}$ satisfies:
\begin{enumerate}
\item It is not biorderable.
\item It does not admit nontrivial actions by $C^1$-diffeomorphisms on the closed interval, $[0,1)$ or the circle.
\end{enumerate}
\end{customcor}
\begin{proof}[Proofs of Corollaries \ref{Corollary LeftOrderable} and \ref{Corollary LocallyIndicable}]
The group $\overline{T}$ is finitely presented, perfect (that is, $\overline{T}=[\overline{T},\overline{T}]$), and left orderable.
It is easy to show that the set of normal subgroups of $\overline{T}$ coincides 
with subgroups that lie in the infinite cyclic center.
It follows that the only nontrivial left orderable quotient of $\overline{T}$ is $\overline{T}$ itself, whence $\overline{T}$ is $lo$-near isolated.
We conclude from Theorem \ref{Theorem: main1} that there is a comeager subset
$\mathcal{X}\subset \mathcal{G}_{lo}$ such that each group $G\in \mathcal{X}$ contains $\overline{T}$ as a subgroup. It follows that such $G$ is not locally indicable, finishing the proof of Corollary \ref{Corollary LeftOrderable}$(1)$.

Set $H=\mathbf{F}_2\ltimes \mathbf{Z}^2$. 
Since $H$ has a solvable word problem and is left orderable, it embeds in a simple subgroup of a finitely presented left orderable group. Therefore, $H$ is $lo$-near isolated and hence embeds in every enumerated group belonging to a certain comeager subset of $\mathcal{G}_{lo}$ by Theorem \ref{Theorem: main1}.
It follows that no group in this comeager subset has the Haagerup property, proving Corollary \ref{Corollary LeftOrderable}$(2)$.

The group $G_0$ is locally indicable (as it is a subgroup of the group of piecewise projective homeomorphisms of the real line).
In \cite{BLT}, it was shown that $G_0$ does non admit a nonabelian action by $C^1$-diffeomorphisms on the closed interval or the circle.
As mentioned above, $G_0$ is an isolated group; the conclusions of Corollary \ref{Corollary LeftOrderable}$(3)$ and Corollary \ref{Corollary LocallyIndicable}$(2)$ follow from Theorem \ref{Theorem: main1}, arguing as before.

The group $BS(1,-1)=\langle f,g\mid f^g=f^{-1}\rangle$ is an example of a locally indicable group which is not biorderable. Recall that biorderability is a closed property. Therefore, Corollary \ref{Corollary LocallyIndicable}$(1)$ follows from  Theorem 
\ref{Theorem: main1}$(2)$. 

Theorem \ref{bludovglass2} asserts that left orderability is a Boone-Higman property. Therefore, part $(4)$ of Corollary \ref{Corollary LeftOrderable} follows from Theorem \ref{Theorem: main2}. 
\end{proof}

\begin{customcor}{\ref{Corollary Kaplansky}}
The following holds:
\begin{enumerate}
\item There is an open dense set $\mathcal{X}\subset \mathcal{G}_{tf}$ such that each enumerated group $G\in \mathcal{X}$ is a counterexample to the unit conjecture.
\item The generic torsion-free enumerated group does not have the unique product property.
\item The generic group with the unique product property is not left orderable.
\item Either the Kaplansky zero divisor conjecture holds or else the generic torsion-free group does not satisfy the zero divisor conjecture.
\item Either the Kaplansky idempotent conjecture holds or else the generic torsion-free group does not satisfy the idempotent conjecture.
\end{enumerate}
\end{customcor}

 \begin{proof}
Gardam in \cite{gardam} proved that the Promislaw group is a counterexample to the Kaplansky unit conjecture.
Since the group is finitely presented and has a solvable word problem, by Theorem \ref{Theorem: main3}, it embeds in every group belonging to a comeager subset of $\mathcal{G}_{tf}$. (We use here the fact that torsion freeness is a Boone Higman property.) Part $(1)$ of the Corollary follows.

Recall that a construction of Rips and Sageev provides finitely generated torsion-free groups without the unique product property (see \cite{Steenbock}). Also, Dunfield discovered examples of finitely generated groups with the unique product property that are not left orderable (see the appendix in \cite{Dunfield}). Finally, recall that the failure of the unique product property, as well as the failure of left orderability, are both open properties. It follows from Theorem \ref{Theorem: main1}$(2)$ that there are comeager sets 
$$\mathcal{X}_1\subset \mathcal{G}_{tf}\qquad 
\mathcal{X}_2\subset \mathcal{G}_{upp}$$ 
such that no group in $\mathcal{X}_1$ has the unique product property and no group in $\mathcal{X}_2$ is left orderable. This proves parts $(2)$ and $(3)$ of the Corollary.
Parts $(4)$ and $(5)$ follow from a very similar argument using Theorem \ref{Theorem: main1}$(2)$.
\end{proof}

\begin{customcor}{\ref{Corollary Sofic}}
For any $P\in \mathcal{P}$, either all groups in $\mathcal{G}_P$ are sofic or else the generic group in $\mathcal{G}_P$ is nonsofic.
\end{customcor}

 \begin{proof}
 The proof follows from the fact that soficity is an open property and a direct application of Theorem \ref{Theorem: main1}$(2)$.
\end{proof}

\section{More about genericity}
In this section, we reiterate the use the following convention to make the distinction between an isomorphism type and its enumeration more precise. We shall denote an isomorphism type with letters such as $G$ and $H$ and chosen enumerations for the respective groups as $\bd G$ and $\bd H$ respectively.
\subsection{Applications of the Baire alternative}

Recall that a subset of a topological space is said to be \textbf{Baire measurable} if it belongs to the smallest $\sigma$-algebra containing the open sets and the meager sets.  The \textbf{Baire alternative} states that a Baire measurable subset of a topological space is either meager or there is a nonempty open set where it is comeager; if the topological space is a \textbf{Baire space} (that is, a topological space for which the conclusion of the Baire category theorem holds, e.g. Polish spaces), then exactly one of the two alternatives hold.  We investigate consequences of this fact in our context. 

\begin{prop}
Suppose that $\cal C$ is a saturated subspace of $\cal G$ such that the set of isomorphism types $\frak C$ of $\cal C$ is closed under direct sums.  Further suppose that $\cal D$ is a saturated, Baire measurable subset of $\cal D$.  Then either $\cal D$ is meager in $\cal C$ or comeager in $\cal C$.     
\end{prop}

\begin{proof}
Suppose that $\cal D$ is not meager in $\cal C$, whence $\cal D$ is comeager in a nonempty open set $[\Sigma(\vec a)]_\cal C$.  Since $\cal D$ is saturated, $\cal D$ is comeager in $\bigcup_{\vec b\in \mathbf{N}^{n_{\Sigma}}}[\Sigma(\vec b)]_{\cal C}$, which is itself comeager in $\cal C$ since $\frak C$ is closed under direct sums.  It follows that $\cal D$ is comeager in $\cal C$, as desired.
\end{proof}

\begin{prop}
If $\varphi(\vec x)$ is an $L_{\omega_1,\omega}$-formula and $\vec a\in \mathbf{N}^{n_\varphi}$, then $$\{\bd G \in \cal G \ : \ \bd G\models \varphi(\vec a)\}$$ is a Borel subset of $\cal G$.
\end{prop}

\begin{proof}
A straightforward induction on the complexity of formulae.
\end{proof}

\begin{cor}\label{genericsentence}
Suppose that $\cal C$ is a saturated, Baire subspace of $\cal G$ such that its set of isomorphism types $\frak C$ is closed under direct sums and $\varphi$ is an $L_{\omega_1,\omega}$-sentence.  Then exactly one of $\{\bd G\in \cal C \ : \ \bd G\models \varphi\}$ or $\{\bd G\in \cal C \ : \ \bd G\models \neg \varphi\}$ is comeager in $\cal C$.
\end{cor}

In the rest of this subsection, we give some examples of the utility of the previous ideas.
Recall that a group $G$ is \textbf{inner amenable} if it admits a conjugation-invariant finitely additive probability measure not concentrating on the identity.

We can apply the above the conclude the proof of Proposition \ref{remark VNDay}.  We remind the reader of the statement:

\begin{customprop}{\ref{remark VNDay}}

\

\begin{enumerate}
\item The generic enumerated group without $\mathbf{F}_2$ subgroups is inner amenable.
\item The generic left orderable enumerated group without $\mathbf{F}_2$ subgroups is inner amenable. 
\end{enumerate}
\end{customprop}

\begin{proof}
For each $n$, let $\sigma_n$ be the sentence $$\forall x_1\cdots\forall x_n\exists y (\bigwedge_{i=1}^n x_iy=yx_i \wedge y\not=e).$$  By Corollary \ref{genericsentence}, $\{\bd G \in \cal \Gsm \ : \ \bd G\models \bigwedge_n\sigma_n\}$ is either meager or comeager in $\Gsm$.  However, this set is clearly dense in $\Gsm$, whence it must be comeager.  It remains to note that all elements $G$ of this set are inner amenable. Indeed, write $G$ as an increasing union of finite subsets $F_n$ and let $g_i\in \bd G\setminus\{e\}$ be an element that commutes with each element in $F_i$. Now consider $h_i= \delta_{g_i}\in \ell^1(G)$, the characteristic function of $g_i$. Any weak* limit of the $h_i$ is a conjugation invariant mean.
\end{proof}

Let $\cal G_{fg}$ denote the saturated subspace of $\cal G$ consisting of finitely generated enumerated groups.

\begin{prop}\label{fgmeager}
$\cal G_{fg}$ is a meager Borel (in fact, $\Sigma_3^0$) subset of $\cal G$.
\end{prop}

\begin{proof}
First note that $$\cal G_{fg}=\bigcup_{\vec a\in \bf N^{<\bf N}}\bigcap_{b\in \bf N}\bigcup_{w(\vec x)}[w(\vec a)=b]$$ which is a $\Sigma_3^0$ subset of $\cal G$.  Let $\sigma_n$ be the sentence $$\forall x_1\cdots\forall x_n\exists y (\bigwedge_{i=1}^n x_iy=yx_i \wedge y\not=e)$$  

We have already seen in the proof above that $\{\bd G\in \cal G \ : \ \bd G\models \bigwedge_n \sigma_n\}$ is a comeager set in $\cal G$, and from Remark \ref{Remark: Simple} we know that $\{\bd G\in \cal G \ : \ G \text{ is simple}\}$ is also a comeager set in $\cal G$. It remains to notice that the comeager set
$$\{\bd G\in \cal G \ : \ \bd G\models \bigwedge_n \sigma_n\}\cap \{\bd G\in \cal G \ : \ G \text{ is simple}\}$$ consists of groups that are not finitely generated.
\end{proof}

\subsection{Generic sets and model-theoretic forcing}

In this section, we provide a connection between Question \ref{Question: main}(2), and model theoretic forcing. First, we observe the following remark:
\begin{remark}\label{scottsentence}
By Corollary \ref{genericsentence}, letting $\varphi$ be a Scott sentence\footnote{Given a countable group $G$, a Scott sentence for $G$ is a $L_{\omega_1,\omega}$-sentence $\sigma_G$ such that, for any countable group $H$, $H\models \sigma_G$ if and only if $G\cong H$.}, we see that the set in the previous question is either meager or comeager.
\end{remark}

We now explain why the answer to \ref{Question: main}$(2)$ is negative when we consider the space $\cal G$ itself.  In order to do so, it will help us to rephrase this in the language of model-theoretic forcing via the presentation in \cite{hodges}.  

The connection we now describe is in fact hinted at in \cite{hodges} (see Exercises 4-6 from Section 2.2).

For the rest of this section, we fix a saturated subspace $\cal C$ of $\cal G$.  We consider a two-player game where the players take turns playing $\cal C$-systems with the requirement that each system played extends the previous players turn.  The players play countably many rounds.  When the game is over, the players have constructed an infinite system of equations and inequations. We call a play of the game \textbf{definitive} if, for all $m,n\in \bf N$, there is $k\in \bf N$ such that the equation $x_m\cdot x_n=x_k$ appears in the final system.  In what follows, we always assume that the play of the game is definitive.\footnote{In \cite{hodges}, the definitiveness requirement is not present.  However, in the terminology used there, being definitive is an enforceable property and thus, for our purposes, there is no loss of generality in assuming that the plays are definitive.}  In this case, at the end of the game, the players have described an enumerated group, called the \textbf{compiled group}.\footnote{Without the definitive requirement, the compiled group would merely be the group generated by $\bf N$ subject to the relations given by the equations of the final system.}

We call a property $P$ of enumerated groups \textbf{$\cal C$-enforceable} if there is a strategy for player II that ensures that the compiled group always has property $P$.  A useful fact is the \textbf{Conjunction Lemma} (see \cite[Lemma 2.3.3(e)]{hodges}, which states that a countable conjunction of $\cal C$-enforceable properties is $\cal C$-enforceable.

\begin{prop}
Suppose that $\cal C$ is a saturated Polish subspace of $\cal G$.  Then the property of being in $\cal C$ is a $\cal C$-enforceable property.
\end{prop}

\begin{proof}
Since $\cal C$ is a Polish subspace of $\cal G$, there are open subsets $U_n$ of $\cal G$ such that $\cal C=\bigcap_n U_n$.  By the Conjunction Lemma, it suffices to show, for each $n$, that the property of belonging to $U_n$ is $\cal C$-enforceable.  Suppose player I opens with the $\cal C$-system $\Sigma$.  Fix $\bd G\in [\Sigma]_\cal C$.  Since $\bd G\in U_n$ and $U_n$ is open, there is a $\cal C$-system $\Delta$ such that $\bd G\in [\Delta]_\cal C\subseteq U_n$.  Then player II responds with the $\cal C$-system $\Sigma\cup \Delta$.  It follows that the compiled group belongs to $[\Delta]_\cal C$ and thus to $U_n$, as desired.
\end{proof}

Given a property $P$ of enumerated groups and $\mathcal{C}\subseteq \mathcal{G}$ a saturated, Baire measurable subset, 
we define $$\mathcal{C}_P=\{G\in \mathcal{C}\mid G\text{ has the property }P\}.$$
we say that $P$ is \textbf{invariant} if $\mathcal{C}_P$ is saturated and we say that $P$ is \textbf{Baire measurable} if $\mathcal{C}_P$ is a Baire measurable subset of $\cal C$. Note that any property of countable groups will be invariant (since it is a property of an isomorphism type, not a specific enumeration.)

Here is the connection between Baire category and enforceability:

\begin{thm}\label{baireenforceable}
Suppose that $\cal C$ is a saturated, Baire subspace of $\cal G$ and that $P$ is an invariant Baire measurable property.  Then $\mathcal{C}_P$ is a comeager subset of $\cal C$ if and only if $P$ is a $\cal C$-enforceable property.
\end{thm}

\begin{proof}
First suppose that $\cal C_P$ is a comeager subset of $\cal C$.  Since $\cal C$ is a Baire space, there is a countable collection of dense open sets $U_n\subseteq \cal G$ such that $\bigcap_n U_n\subseteq \cal C_P$.  In order to show that $P$ is $\cal C$-enforceable, it suffices, for every $n$, to show that the property "$\bd H\in U_n$'' is a $\cal C$-enforceable property.  Towards this end, suppose that player I opens with the system $\Sigma$.  Since $U_n$ is dense, there is a group $\bd H\in [\Sigma]_\cal C\cap U_n$.  Since $U_n$ is open, there is a system $\Delta$ such that $\bd H\in [\Delta]_\cal C\subseteq U_n$.  Let player II respond with $\Sigma\cup \Delta$.  Then the compiled group will belong to $U_n$, as desired.

Now suppose that $P$ is a $\cal C$-enforceable property.  If $\cal C_P$ were meager, then $\cal C\setminus \cal C_{P}$ is comeager; since $P$ is invariant, it follows that the negation of property $P$ is $\cal C$-enforceable, which is a contradiction.  
\end{proof}

We shall need the following notion.

\begin{defn}\label{Definition: Existentially closed}
If $G$ is a subgroup of $H$, we say that $G$ is \textbf{existentially closed (or, e.c., for short) in $H$} if, for any finite system $\Sigma(\vec x,\vec y)$ and any $\vec a\in G$, if there is a solution to $\Sigma(\vec a,\vec y)$ in $H$, then there is a solution to $\Sigma(\vec a,\vec y)$ in $G$. If $\frak C$ is the set of isomorphism types of a class $\mathcal{C}\subseteq \mathcal{G}$, we say that $G\in \frak C$ is \textbf{existentially closed for $\frak C$} if $G$ is e.c. in $H$ for every $H\in \frak C$ containing $G$ as a subgroup.
\end{defn}

We now present some elementary facts about this notion.

\begin{prop}\label{Prop: EC}
Let $\frak C$ be a class of isomorphism types.
Then the following holds.
\begin{enumerate}

\item If $\frak C$ is closed under direct limits, then any element of $\frak C$ is a subgroup of a group that is e.c. for $\frak C$.  
\item Suppose that any two elements of $\frak C$ can be embedded into a common element of $\frak C$ (e.g. when $\frak C$ is closed under direct products).  Then any e.c. element of $\frak C$ is locally universal for $\frak C$.
\end{enumerate}
More generally, if $G$ is a subgroup of $H$, then $G$ is e.c. in $H$ if and only if $H$ embeds into an ultrapower of $G$ in such a way that the restriction to $G$ is the diagonal embedding of $G$ into its ultrapower.
\end{prop}
Let $\mathcal{C}\subseteq \mathcal{G}$ be a saturated subset and let $\frak C$ be its underlying class of isomorphism types. We let $\frak C_{ec}$ denote the collection of e.c. objects in $\frak C$ and we set $\mathcal{C}_{\ec}\subseteq \mathcal{C}$ denote the set of enumerations of groups in $\frak C$.

\begin{lem}\label{Lemma: DirectLimitsDense}
Suppose that $\frak C$ is closed under direct limits.  Then $\cal C_{\ec}$ is dense in $\cal C$.
\end{lem}

\begin{proof}
Suppose that $[\Sigma(\vec a)]_\cal C$ is a nonempty basic open subset of $\cal C$ and $\bd G\in [\Sigma(\vec a)]_\cal C$. By Proposition \ref{Prop: EC}, we can find a group $H\supseteq G$ which is an e.c. element of $\frak C$.  Fix an enumeration of $\bd H$ that agrees with $\bd G$ on $\vec a$.  Then $\bd H\in [\Sigma(\vec a)]_\cal C\cap \cal C_{\ec}$.
\end{proof}

Unlike the case of locally universal groups, we do not know if the set of e.c. elements of a given class is comeager.  However, we do have such a result in the following context:

\begin{lem}\label{ecgen}
For any universal theory $T$ extending the theory of groups, if we set $\frak C:=\frak C_T$, then $\cal C_{\ec}$ is comeager in $\cal C$.
\end{lem}

\begin{proof}
By an application of Proposition \ref{closeduniversal}, since $\cal C$ is closed, we can deduce that $\frak C$ is closed under direct limits. Hence by Lemma \ref{Lemma: DirectLimitsDense}, $\mathcal{C}_{\ec}$ is dense in $\cal{C}$. So we only need to show that $\cal C_{\ec}$ is $G_\delta$ in $\cal C$.  Fix a system $\Sigma(\vec x,\vec y)$ and $\vec a\in \bf N$.  Set $$\cal Y:=\cal Y_{\Sigma,\vec a,\cal C}:=\{\bd G\in \cal C \ : \ \bd G\models \exists \vec y \Sigma(\vec a,\vec y)\}$$ and note that $\cal Y=\bigcup_{\vec b\in \bf N}[\varphi(\vec a,\vec b)]_\cal C$, whence is open.  We claim that $$\cal Z:=\cal Z_{\Sigma,\vec a,\cal C}=\{\bd G \in \cal C \ : \ \text{ for all }H\supseteq G \text{ with }H\in \cal C, \ H\models \forall \vec y\neg\Sigma(\vec a,\vec y)\}$$ is also an open subset of $\cal C$. 

Indeed, suppose that $\bd G\in \cal Z$ and let $\{[\Sigma_n(\vec b_n)]_\cal C\}_{n\in \mathbf{N}}$ denote a countable neighborhood base of $\bd G$.  Suppose, towards, a contradiction, that for each $n$, there is $$\bd H_n\in [\Sigma_n(\vec b_n)]_\cal C\cap \bigcup_{\vec c\in \bf N}[\Sigma(\vec a,\vec c)]_\cal C$$  Fixing a nonprincipal ultrafilter $\u$ on $\bf N$, an argument similar to (but slightly more elaborate than) that occurring in the proof of Lemma \ref{locunivequiv} shows that $G$ embeds into $\prod_\u H_n$.  Since $$\prod_\u H_n\models T\cup \{\exists \vec y\Sigma(\vec a,\vec y)\}$$ this contradicts the fact that $\bd G\in \cal Z$.  Consequently, for some $n$, we have that $$[\Sigma_n(\vec b_n)]_\cal C\subseteq \bigcap_{\vec c\in \bf N}[\neg\Sigma(\vec a,\vec c)]_\cal C$$  Finally, we note that $$[\Sigma_n(\vec b_n)]_\cal C\subseteq \cal Z$$  Indeed, if $\bd H\in [\Sigma_n(\vec b_n)]_\cal C$ and $K\supseteq H$ belongs to $\cal C$, then by fixing an enumeration $\bd K$ of $K$ for which $\bd K\in [\Sigma_n(\vec b_n)]_\cal C$, we have that $\bd K\models \forall \vec y\neg\Sigma(\vec a,\vec y)$.

 It remains to note that
$$\cal C_{\ec}=\bigcap_{\Sigma,\vec a}\left(Y_{\Sigma,\vec a,\cal C}\cup \cal Z_{\Sigma,\vec a,\cal C}\right).$$
\end{proof}

The following corollary of Lemma \ref{ecgen} is an immediate generalization of an argument originally due to Macintyre in the case that $T$ is the theory of groups itself (see, for example, \cite[Theorem 3.4.6]{hodges}):

\begin{cor}\label{ecsolv}
Suppose that $T$ is a recursively enumerable universal theory extending the theory of groups and $H$ is a finitely generated group without solvable word problem.  Then there is an e.c. model $G$ of $T$ into which $H$ does not embed.
\end{cor}

\begin{proof}
Let $\vec h$ denote a finite generating set for $H$ and let $\Phi_+(\vec x)$ and $\Phi_-(\vec x)$ denote the set of equations and inequations satisfied by $\vec h$ in $H$.  Set $\Phi(\vec x)=\Phi_+(\vec x)\cup \Phi_-(\vec x)$.  Letting $n$ denote the length of $\vec h$, by Lemma \ref{ecgen} and the Baire Category Theorem, it suffices to show that, for any $\vec a\in \bf N^n$, the set $\bigcup_{\varphi\in\Phi}[\neg \varphi(\vec a)]_{\cal C_T}$ is dense, for then any $G$ belonging to the comeager set $\cal C_{ec}\cap \bigcap_{\vec a\in \bf N^n}\bigcup_{\varphi\in\Phi}[\neg \varphi(\vec a)]_{\cal C_T}$ is as desired.  Suppose, towards a contradiction, that $\bigcup_{\varphi\in\Phi}[\neg \varphi(\vec a)]_{\cal C_T}$ is not dense, whence there is some $\cal C_T$-system $\Sigma(\vec x)$ for which $[\Sigma(\vec a)]_{\cal C_T}\subseteq \bigcap_{\varphi\in\Phi}[\varphi(\vec a)]_{\cal C_T}$.  In other words, for all equations $\varphi$, we have $\varphi\in \Phi_+(\vec x)$ if and only if $T\models \forall \vec x(\Sigma(\vec x)\rightarrow \varphi(\vec x))$, implying that the set of all equations true of $\vec h$ in $H$ is recursively enumerable.  The same argument shows that the set of inequations true of $\vec h$ in $H$ is also recursively enumerable, whence $H$ has solvable word problem, leading to a contradiction. 
\end{proof}
The following is \cite[Corollary 3.4.3]{hodges}:
\begin{prop}\label{ecenforceable}
Suppose that $\cal C$ is a closed saturated subset of $\cal G$.  Then the property of being e.c. for $\cal C$ is $\cal C$-enforceable.
\end{prop}

\begin{prop}\label{AVE}
Suppose that $\cal C$ and $\cal D$ are saturated Polish subspaces of $\cal G$ with $\cal D\subseteq \cal C$.  Further suppose that there is $\bd G\in \cal D$ such that $G$ is locally universal for $\frak C$.  Then the property of belonging to $\cal D$ is $\cal C$-enforceable.
\end{prop}

\begin{proof}
Let $U_n$ be open subsets of $\cal G$ such that $\cal D=\bigcap_n U_n$.  It suffices to show, for each $n$, that belonging to $U_n$ is $\cal C$-enforceable.  Suppose that player I opens with the the $\cal C$-system $\Sigma$.  Since $\bd G$ is locally universal for $\frak C$, we have that $\bd G\models \Sigma$.  Since $\bd G\in U_n$ and $U_n$ is open, there is a $\cal D$-system $\Delta$ such that $\bd G\in [\Delta]_\cal C\subseteq U_n$.  If player II responds with the $\cal C$-system $\Sigma\cup \Delta$, we have that the compiled group belongs to $U_n$, as desired.
\end{proof}


\begin{cor}
Suppose that $\cal C$ is a saturated Baire measurable subspace of $G$.  Given $G\in \frak C$, the set $\{\bd H\in \cal C \ : \ H\cong G\}$ is comeager in $\cal C$ if and only if the property of being isomorphic to $G$ is $\cal C$-enforceable.
\end{cor}

When the equivalent conditions of the following corollary are satisfied, we call $G$ the \textbf{$\cal C$-enforceable group}.

The following fact follows from \cite[Theorem 4.2.6 and Exericse 4.2.2(a)]{hodges}:

\begin{fact}\label{embedsintoallec}
Suppose that $P$ is an axiomatizable property of groups that is closed under direct limits.  Further suppose that the $\cal G_P$-enforceable group $G$ exists.  Then $G$ embeds into every $\cal G_P$-e.c. group.
\end{fact}


We are now ready to prove:

\begin{customthm}{\ref{main theorem 4}}
If $P$ is a strongly undecidable, recursively axiomatizable, Boone-Higman property that is closed under subgroups, then there is no comeager isomorphism class in $\mathcal{G}_P$.  
\end{customthm}

\begin{proof}
Suppose, towards a contradiction, that there is a comeager isomorphism class in $\mathcal{G}_P$, that is, the $\cal G_P$-enforceable group $G$ exists.  Since $P$ is a strongly undecidable Boone-Higman property closed under subgroups, Theorem \ref{Theorem: main3} and the Baire category theorem imply that $G$ contains a finitely generated subgroup $H$ with property $P$ that has an unsolvable word problem.  By Corollary \ref{ecsolv}, there is an e.c. element of $\cal G_P$ into which $H$ does not embed.  On the other hand, $G$ embeds into all e.c. elements of $\cal G_P$ by Fact \ref{embedsintoallec}, leading to a contradiction.
\end{proof}

We document here A. Darbinyan's proof that left orderability is strongly undecidable. 

\begin{prop}\label{Arman's proof}
Left orderability is strongly undecidable.
\end{prop}
\begin{proof}
Fix recursively inseparable recursively enumerable subsets $M,N\subseteq \mathbf{N}$.  (Recall that this means that there is no recursive set $K$ such that $M\subset K$ and $N\cap K=\emptyset$; see \cite{rogers} Section 7.7.) Let $A_0$ be the free abelian group on countably many generators $a_i$. Consider the following abelian quotient $A$ of $A_0$: $$A=\langle  a_1,a_2,\hdots | a_1=a_i \text{ if $i\in M$ and $a_2=a_j$ if $j\in N$}  \rangle.$$ Let $\pi_0$ be the quotient map from $A_0$ to $A$. It was shown in \cite[Theorem 3 (a), (c), (f)]{arman2} that $A$ can be embedded into a $2$-generated, left orderable, recursively presented group $G_0$ by an embedding $\Phi$ such that the map $i\mapsto \Phi(a_i)$ is computable. (By computability, we mean the existence of an algorithm that takes as an input $i\in \mathbf{N}$, 
and produces as an output a word in the finite generating set of $G_0$ that represents the image of $a_i$ in $G_0$.)

Since every recursively enumerable left orderable group embeds in a finitely presented left orderable group (by Theorem \ref{bludovglass}), we can fix a finitely presented left-orderable group $G$ that contains a copy of $G_0$. Moreover, the inclusion $G_0\hookrightarrow G$ is computable as $G$ is finitely presented. Set $F=\{\Phi(a_1),\Phi(a_2)\}\subseteq G$. Then, for the canonical surjection $id_G$ on $G$ is injective when restricted to $F$. It remains to verify that, for any surjection $\pi$ from $G$ to a group $H$ that is injective when restricted to $F$, we have that $H$ has unsolvable word problem. Indeed, if the word problem of $H$ were solvable, then the set $K:=\{i\in \mathbf{N}: (\pi\circ\Phi\circ\pi_0)(a_i)=(\pi\circ\Phi\circ \pi_0)(a_1)\}$ is recursive, contains $M$, and is disjoint from $N$, contradicting the recursive inseparability of $M$ and $N$. 
\end{proof}

\begin{prop}\label{leftrec}
There is a recursively axiomatizable theory $T_{lo}$ whose models are precisely the left orderable groups.
\end{prop}

\begin{proof}
The proof hinges on a simple reformulation of the algebraic criteria for left orderability given in Fact \ref{orderabilityopen}, namely the group $G$ is left orderable if and only if, given finitely many $g_1,\ldots,g_n\in G\setminus\{e\}$, there is $E=(\epsilon_1,\ldots,\epsilon_n)\in \{1,-1\}^n$ so that the semigroup generated by $g_1^{\epsilon_1},\ldots,g_n^{\epsilon_n}$ does not contain the identity.  Given $m\in \bf N$, set $S(g_1,\ldots,g_n,E,m)$ to be the set of words in $g_1^{\epsilon_1},\ldots,g_n^{\epsilon_n}$ of length at most $m$.  The Pigeonhole Principle then implies that a group $G$ is left orderable if and only if, given any finitely many $g_1\ldots,g_n\in G\setminus\{e\}$ and any $m\in \bf N$, there is $E=(\epsilon_1,\ldots,\epsilon_n)\in \{1,-1\}^n$ such that $e\notin S(g_1,\ldots,g_n,E,m)$.  Consequently, we can simply let $T_{lo}$ consist of the sentences $\sigma_{m,n}$ for $m,n\in \bf N$, where $\sigma_{m,n}$ is the sentence
$$\forall x_1\cdots \forall x_n\left(\bigwedge_{i=1}^n x_i\not=e\rightarrow \bigvee_{E\in \{1,-1\}^n}\bigwedge_{w\in S(x_1,\ldots,x_n,E,m)}w(x_1,\ldots,x_n)\not=e\right).$$
\end{proof}





With regard to Question \ref{Question: main}(2) in general, we point out the following dichotomy, which is immediate from the Baire category theorem:

\begin{prop}
For any property $P$ closed under direct sums, in $\mathcal{G}_P$, either there exists a comeager isomorphism class, or every comeager set contains uncountably many isomorphism classes. 
\end{prop}
\begin{proof}
We know from Remark \ref{scottsentence} that every isomorphism class is either meager or comeager. Suppose there exists no comeager isomorphism class, i.e, every isomorphism class is meager. Then by the Baire category theorem, no comeager set can be written as a union of countably many isomorphism classes (as they are all meager).
\end{proof}

\subsection{Amenable groups satisfying a law}
We end this section by investigating questions emerging in the context of groups satisfying a law.
Let $w(\vec x)$ be a freely reduced word in the letters $\{x_1^{\pm},...,x_n^{\pm}\}$ and $\vec x=(x_1,...,x_n)$.
Recall that we defined the Polish space of \emph{enumerated groups satisfying the law $w$}, or $\mathcal{G}_w$, as the set  
$$\cal G_w=\bigcap_{\vec a\in \mathbf{N}^{n}}[w(\vec a)=e].$$
We denote the set of isomorphism types in $\mathcal{G}_w$ as $\frak{G}_w$.
Moreover, we recall that $\mathcal{G}_{am,w}=\mathcal{G}_{am}\cap \mathcal{G}_{w}$ and $\frak{G}_{am,w}=\frak{G}_{am}\cap\frak{G}_{w}$.

Some laws imply amenability, e.g. groups satisfying the law $[x,y]=e$ are abelian and hence amenable.  We call a nontrival word $w$ \emph{amenable} if $\frak G_w$ consists only of amenable groups. 
Otherwise, the word $w$ is called \emph{nonamenable}.
The following question is a key consideration.

\begin{question}\label{genericlaw}
Suppose $w$ is a nonamenable law.  Is the generic group satisfying the law $w=e$ nonamenable?

\end{question}

The following is clear:

\begin{lem}
For any word $w$, $\frak G_w$ is closed under direct products.
\end{lem}

Since satisfying the law $w=e$ is clearly expressible by a single universal sentence, Lemma \ref{ecgen} immediately implies:

\begin{lem}
For any word $w$, $(\cal G_w)_{\ec}$ is comeager in $\cal G_w$.
\end{lem}


Recall that a group $G$ is \emph{uniformly amenable} if there is a function $f:\bf N\to \bf N$ such that, for any finite $F\subseteq \bf N$ and any $n\geq |F|$, there is $K\subseteq G$ with $|K|\leq f(n)$ such that $K$ is a $(F,\frac{1}{n})$-Folner set for $F$.  The following straightforward fact was observed by Keller in \cite{keller}:

\begin{fact}
$G$ is uniformly amenable if and only some (equiv. every) ultrapower of $G$ is amenable.
\end{fact}

The following fact is also straightforward:

\begin{fact}\label{lawlessfact}
$G$ is lawless if and only if $\bf F_2$ embeds into some (equiv. every) nonprincipal ultrapower of $G$.  
\end{fact}

In other words, $G^\u$ is small if and only if $G$ satisfies some nontrivial word.  It is unknown whether or not von Neumann's problem has a positive solution for ultrapowers, that is, the following question is open:

\begin{question}\label{vndultra}
If $G$ is an amenable group such that $G^\u$ is small, must $G^\u$ be amenable?  In other words, if $G$ is an amenable group that satisfies a nontrivial law, must $G$ be uniformly amenable?
\end{question}

 The following lemma is clear:

\begin{lem}
If $w$ is an amenable word, then every element of $\frak G_w$ is uniformly amenable.
\end{lem}  

Consequently, Question \ref{vndultra} is really only interesting when $G$ satisfies a nonamenable law. 

Following typical model-theoretic nomenclature, we call a group \emph{pseudoamenable} if it is elementarily equivalent to an ultraproduct of amenable groups.

\begin{prop}
For a given word $w$, the following are equivalent:
\begin{enumerate}
    \item If $G\in \frak G_{\operatorname{am},w}$, then $G^\u$ is amenable.
    \item $\frak G_{\operatorname{am},w}$ is an elementary class.
    \item $\Gamw$ is closed in $\cal G_w$.
    \item If $G\in \frak G_w$ is pseudoamenable, then $G$ is amenable.
\end{enumerate}
\end{prop}

\begin{proof}
For (1) implies (2), fix a family $(G_i)_{i\in I}$ from $\mathfrak G_{\operatorname{am},w}$ and an ultrafilter $\u$ on $I$.  We must show that $\prod_\u G_i$ is also amenable.  However, setting $G:=\bigoplus_{i\in I}G_i$, we have that $G\in \frak G_{\operatorname{am},w}$, whence $G^\u$ is amenable by (1).  Since $\prod_\u G_i$ embeds into $G^\u$, we have that $\prod_\u G_i$ is amenable, as desired.


(2) implies (1) is clear.  Since $\frak G_{\operatorname{am},w}$ is closed under subgroups, the equivalence of (2) and (3) follows from Corollary \ref{closeduniversal}.

(4) implies (1) follows from the fact that $G^\u$ is a pseudoamenable member of $\frak G_w$ whenever $G$ is an amenable member of $\frak G_w$.  Now suppose that (1) holds and that $G\in \frak G_w$ is pseudoamenable.  By assumption, there is a family $(G_i)_{i\in I}$ of amenable groups and an ultrafilter $\u$ on $I$ such that $G\equiv \prod_\u G_i$.  Since $G$ satisfies the law $w=e$, we have $G_i$ also satisfies the law $w=e$ for $\u$-almost all $i\in I$.  By replacing, for each $i\in \bf N\setminus I$, $G_i$ with an amenable group satisfying $w$, we may as well assume that each $G_i$ satisfies $w$.  Let $H=\bigoplus_{i\in \bf N}G_i$, an amenable group satisfying $w$.  Since $G\equiv \prod_\u G_i$, we have that every system with a solution in $G$ has a solution in $H$.  Thus, by Lemma \ref{locunivequiv}, we have that $G$ embeds into an ultrapower of $H$.  By (1), this ultrapower of $H$ is amenable, whence so is $G$.
\end{proof}

Motivated by item (3) in the previous proposition, we call a word for which the items in the previous proposition hold a \emph{closed word}.  Thus, Question \ref{vndultra} asks whether or not all words are closed.

\subsection{The generic element of $\cal G_{\operatorname{am},w}$}
Since groups satisfying a nontrivial law are automatically small, the results in the previous section motivate us to ask the following:

\begin{question}\label{genericlaw}
Suppose $w$ is a nonamenable law.  Is the generic group satisfying the law $w=e$ nonamenable?

\end{question}

We will need the following fact.
\begin{fact}\label{eclocuniv}
Let $\frak C$ be a set of isomorphism types of countable groups.
Suppose that any two elements of $\frak C$ can be embedded into a common element of $\frak C$ (e.g. when $\frak C$ is closed under direct products).  Then any e.c. element of $\frak C$ is locally universal for $\frak C$.
\end{fact}

The following theorem shows us that all possible ways of making the word generic precise in the previous question lead to the same conclusion:

\begin{thm}
The following are equivalent:
\begin{enumerate}
    \item Every locally universal element of $\cal G_w$ is nonamenable.
    \item Every e.c. element of $\cal G_w$ is nonamenable.
    \item $\cal G_w\setminus \cal G_{\operatorname{am},w}$ is comeager in $\cal G_w$.
    \item Being nonamenable is a $\cal G_w$-enforceable property.
\end{enumerate}
\end{thm}

\begin{proof}
(1) implies (2) follows from Fact \ref{eclocuniv}.  (2) implies (3) follows from Lemma \ref{ecgen}. The equivalence of (3) and (4) follows from Theorem \ref{baireenforceable}. Finally, (4) implies (1) follows from Proposition \ref{AVE} and the fact that the amenable groups form a Polish space.
\end{proof}






The connection between Question \ref{vndultra} and the amenability of the generic element of $\cal G_w$ is the following:

\begin{cor}
If $w$ is a closed nonamenable word, then $\cal G_w\setminus \cal G_{\operatorname{am},w}$ is comeager in $\cal G_w$.
\end{cor}

\begin{proof}
Suppose that $w$ is a closed word and that $H$ is an amenable group that is locally universal for $\frak G_w$.  Since $w$ is closed, every ultrapower of $H$ is also amenable, whence so is every element of $\frak G_w$ since $H$ is locally universal for $\frak G_w$.  Consequently, $w$ is an amenable word.
\end{proof}

\begin{question}
Does the converse to the previous corollary hold?
\end{question}

\subsection{A test case}

We now consider one case where we might be able to establish that the generic element of $\cal G_w$ is nonamenable.

For sufficiently large odd $n$, we let ${\gos}_n$ denote the group constructed by Olshanskii and Sapir in \cite{olshanskiisapir}.  We note that ${\gos}_n$ is a finitely presented, small, nonamenable group.  Moreover, ${\gos}_n$ satisfies the law $w_n:=[x,y]^n=e$ and contains the free Burnside group $B(2,n)$ of exponent $n$.\footnote{Recall that $B(2,n)$ is the group generated by $x$ and $y$ subject to the relations $w^n=e$ for all nontrivial words $w=w(x,y)$.  For $n$ sufficiently large and odd, $B(2,n)$ is infinite.}


We believe that the following question is still open.

\begin{question}\label{burnside}
For sufficiently large odd $n$, is $B(2,n)$ residually amenable?
\end{question}

The connection with the above discussion is the following:

\begin{thm}
Either $B(2,n)$ is residually amenable or else $\cal G_{w_n}\setminus \cal G_{\operatorname{am},w_n}$ is comeager in $\cal G_{w_n}$.
\end{thm}

\begin{proof}
If there is an amenable group that is locally universal for $\frak G_{w_n}$, then ${\gos}_n$ is residually amenable, whence so is $B(2,n)$.
\end{proof}

In \cite{weiss}, Weiss asked if the free Burnside groups $B(m,n)$ are sofic.  Since this still remains an open question\footnote{At least to the best of our knowledge.}, we believe that either Question \ref{burnside} is still open or else it has a negative answer, for residually amenable groups are sofic.


\bibliographystyle{plain}
\bibliography{bib}{}

\end{document}